\documentclass{amsart}
\usepackage{amsfonts, color}
\usepackage{latexsym}
\usepackage{amssymb}
\usepackage{amsmath}
\usepackage{url}
\newcommand{\R}{\mathbb R}

\newtheorem{thm}{Theorem}[section]
\newtheorem{cor}{Corollary}[section]
\newtheorem{lemma}{Lemma}[section]

\newtheorem{proposition}{Proposition}[section]

\theoremstyle{remark}
\newtheorem*{rmk}{Remark}
\numberwithin{equation}{section}
\newcommand{\restr}[1]{|_{#1}}
\begin{document}

%%%%%%%%%%%%%%%%%%%%%%%%%%%%%%%%%%%%%%%%%%%%%5

\title{On affine invariant and local Loomis-Whitney type inequalities}

\author[D.\,Alonso]{David Alonso-Guti\'errez}
\address{\'Area de an\'alisis matem\'atico, Departamento de matem\'aticas, Facultad de Ciencias, Universidad de Zaragoza, Pedro Cerbuna 12, 50009 Zaragoza (Spain), IUMA}
\email[(David Alonso)]{alonsod@unizar.es}

\author[J.\,Bernu\'es]{Julio Bernu\'es}
\address{\'Area de an\'alisis matem\'atico, Departamento de matem\'aticas, Facultad de Ciencias, Universidad de Zaragoza, Pedro cerbuna 12, 50009 Zaragoza (Spain), IUMA}
\email[(Julio Bernu\'es)]{bernues@unizar.es}

\author[S. Brazitikos]{Silouanos Brazitikos}
\address{School of Mathematics and Maxwell Institute for Mathematical
	Sciences, University of Edinburgh, JCMB, Peter Guthrie Tait Road King's Buildings,
	Mayfield Road, Edinburgh, EH9 3FD, Scotland.}
\email[(Silouanos Brazitikos)]{silouanos.brazitikos@ed.ac.uk}

\author[A. Carbery]{Anthony Carbery}
\address{School of Mathematics and Maxwell Institute for Mathematical
	Sciences, University of Edinburgh, JCMB, Peter Guthrie Tait Road King's Buildings,
	Mayfield Road, Edinburgh, EH9 3FD, Scotland.}
\email[(Anthony Carbery)]{A.Carbery@ed.ac.uk}
\thanks{Partially supported by Spanish grants MTM2016-77710-P and DGA E-64}
\begin{abstract}
We prove various extensions of the Loomis-Whitney inequality and its dual, where the subspaces on which the projections (or sections) are considered are either spanned by vectors $w_i$ of a not necessarily orthonormal basis of $\R^n$, or their orthogonal complements. In order to prove such inequalities we estimate the constant in the Brascamp-Lieb inequality in terms of the vectors $w_i$. Restricted and functional versions of the inequality will also be considered.
\end{abstract}

\date{\today}
\maketitle
\section{Introduction and notation}

The classical Loomis-Whitney inequality \cite{LW} states that for any convex body $K\subseteq\R^n$ (i.e., a compact convex set with non-empty interior) and any orthonormal basis $\{e_i\}_{i=1}^n$, we have that
\begin{equation}\label{eq:ClassicalLoomisWhitney}
|K|\leqslant\prod_{i=1}^n|P_{e_i^\perp}K|^\frac{1}{n-1},
\end{equation}
where $|\cdot|$ denotes the volume (i.e. the Lebesgue measure) in the corresponding subspace and, for any $k$-dimensional linear subspace $H\in G_{n,k}$, $P_H$ denotes the orthogonal projection onto $H$.

In \cite{M}, Meyer proved the following dual inequality:
\begin{equation}\label{eq:MeyersInequality}
|K|\geqslant\frac{(n!)^\frac{1}{n-1}}{n^\frac{n}{n-1}}\prod_{i=1}^n|K\cap e_i^\perp|^\frac{1}{n-1}.
\end{equation}

A remarkable extension of the Loomis-Whitney inequality is provided by the Bollob\'as-Thomason inequality, which was proved in \cite{BT}. Before stating it let us fix some notation and terminology. We denote by $[n]$ the set $\{1,\dots, n\}$. For any $m\geqslant 1$ and subsets $S_1,\dots, S_m\subseteq[n]$, not necessarily distinct, we  say that {\sl $(S_1,\dots, S_m)$ forms a uniform cover of $S\subseteq[n]$ with weights $(p_1,\dots p_m)$, ($p_j>0$ for all $j$)} if for every $i\in S$ we have that
$$
\sum_{j=1}^m p_j\chi_{S_j}(i)=1.
$$

Bollob\'as and Thomason considered the case $S=[n]$ and weights equal to $\frac{1}{k}$. Then, the above condition means that each index $i\in [n]$ appears exactly $k$ times within the family $S_1,\dots ,S_m$. They proved that for any uniform cover $(S_1,\dots, S_m)$  of $[n]$ with equal weights $\left(\frac{1}{k},\dots,\frac{1}{k}\right)$, any orthonormal basis $(e_i)_{i=1}^n$ and subspaces $H_j:=\textrm{span}\{e_k\,:\,k\in S_j\}$, for any {\sl compact} set $K\subseteq\R^n$ we have
\begin{equation}\label{eq:Bollobas-Thomason}
|K|\leqslant \prod_{j=1}^m|P_{H_j}K|^\frac{1}{k}.
\end{equation}
The case $m=n$, $S_j=[n]\setminus{\{j\}}$ and $k=n-1$ recovers the classical Loomis-Whitney inequality \eqref{eq:ClassicalLoomisWhitney}.

Very recently, Liakopoulos \cite{L} proved the following dual Bollob\'as-Thomason inequality: given a fixed orthonormal basis $(e_i)_{i=1}^n$ and any $m\geqslant 1$, if $(S_1,\dots, S_m)$ forms a uniform cover of $[n]$ with equal weights $\left(\frac{1}{k},\dots,\frac{1}{k}\right)$, $d_j=|S_j|$ and $H_j=\textrm{span}\{e_k\,:\,k\in S_j\}$, then for any compact $K\subseteq\R^n$ we have
\begin{equation}\label{eq:LiakopoulosDual}
|K|\geqslant \frac{\prod_{j=1}^m(d_j!)^\frac{1}{k}}{n!}\prod_{j=1}^m|K\cap H_j|^\frac{1}{k}.
\end{equation}
Again, if $m=n$, $S_j=[n]\setminus{\{j\}}$ and $k=n-1$ we obtain Meyer's inequality \eqref{eq:MeyersInequality}.

In \cite{GHP}, the following restricted (or local) version of Loomis-Whitney inequality was proved: given a fixed orthonormal basis $(e_i)_{i=1}^n$, for any convex body $K\subseteq\R^n$ and any $i\neq j$,
\begin{equation}\label{eq:FirstLocalLoomisWhitney}
|P_{H^\perp}K||K|\leqslant \frac{2(n-1)}{n}|P_{e_i^\perp}K||P_{e_j^\perp}K|,
\end{equation}
where  $H=\textrm{span}\{e_i,e_j\}$.
%It is worth pointing out the obvious fact that $\{e_i,e_j\}$ form an orthonormal basis on $H_{ij}=\textrm{span}\{e_i,e_j\}$  and that $S_1=\{i\}$ and $S_2=\{j\}$ form a uniform cover with weights $1$ of $\{i,j\}\subseteq[n]$ and therefore the product $|P_{e_i^\perp}K||P_{e_j^\perp}K|$ can be regarded in any of the following ways
%\begin{itemize}
%\item $|P_{\textrm{span}\{H_1, H^\perp\}}K||P_{\textrm{span}\{H_2, H^\perp\}}K|$
%\item $|P_{\textrm{span}\{\widetilde{H_1}, H^\perp\}}K||P_{\textrm{span}\{\widetilde{H_2}, H^\perp\}}K|$
%\item $|P_{H_1^\perp}K||P_{H_2^\perp}K|$
%\item $|P_{\widetilde{H_1}^\perp}K||P_{\widetilde{H_2}^\perp}K|$
%\end{itemize}
%where $H_i=\textrm{span}\{e_k\,:\,k\in S_i\}$ and $\widetilde{H_i}=\textrm{span}\{e_k\,:\,k\not\in S_i\}$.
Inequality \eqref{eq:FirstLocalLoomisWhitney} has been extended in several different ways. On the one hand, in \cite{BGL}, it was proved that for any distinct vectors $w_1,w_2\in S^{n-1}$, not necessarily orthogonal, if $H=\textrm{span}\{w_1,w_2\}$ then
\begin{equation}\label{eq:LocalLoomisWhitney2NonOrthogonalVectors}
|P_{H^\perp}K||K|\leqslant\frac{2(n-1)}{n\sqrt{1-\langle w_1,w_2\rangle^2}}|P_{w_1^\perp}K||P_{w_2^\perp}K|.
\end{equation}
On the other hand, in the same work \cite{BGL}, the following generalisation of \eqref{eq:FirstLocalLoomisWhitney} was obtained: if $S\subseteq[n]$ has  cardinality $|S|=d$ and $(S_1,\dots, S_m)$ forms a uniform cover of $S$ with the same weights $(\frac{1}{k},\dots,\frac{1}{k})$,  where $m>k$, then for every convex body $K\subseteq\R^n$
$$
|P_{H^\perp}K||K|^{\frac{m}{k}-1}\leqslant\frac{{{n-\frac{kd}{m}}\choose{n-d}}^\frac{m}{k}}{{{n\choose d}}^{\frac{m}{k}-1}}\prod_{j=1}^m|P_{H_j^\perp}K|^\frac{1}{k},
$$
where $H_j=\textrm{span}\{e_k\,:\,k\in S_j\}$. The value of the constant in the latter inequality was improved in \cite{AAGJV} in the case where $m=2$ and $S_1$, $S_2$ are disjoint, which implies $k=1$. Moreover, in \cite{AAGJV} functional versions of these inequalities were proved in the setting of log-concave functions.

Regarding dual versions of restricted Loomis-Whitney inequalities, two situations have been considered: when the convex body $K$ is centred and when the maximal intersection of $K$ with translations of $H^\perp$ is attained at 0.

On the one hand, it was proved in \cite{BGL} that if $S\subseteq[n]$ has cardinality $|S|=d$ and $(S_1,\dots, S_m)$ forms a uniform cover of $S$ with the same weights $(\frac{1}{k},\dots,\frac{1}{k})$, where $m>k$ and the cardinality of $S_j$ is equal to $d_j$, then for every centred convex body $K\subseteq\R^n$
$$
|K\cap H^\perp||K|^{\frac{m}{k}-1}\geqslant\frac{\prod_{j=1}^md_j^\frac{d_j}{k}}{(c_0d)^{d}}\prod_{j=1}^m|K\cap H_j^\perp|^\frac{1}{k},
$$
where $H_j=\textrm{span}\{e_k\,:\,k\in S_j\}$.

On the other hand, it was proved in \cite{AAGJV} that if $S\subseteq[n]$ has cardinality $|S|=d$ and $S_1$ and $S_2$ are two disjoint subsets of $S$, with cardinality $d_1$ and $d_2$ respectively, forming a uniform cover of $S$ (therefore with weights equal to 1), and if $H=\textrm{span}\{e_k\,:\,k\in S\}$, then for every convex body $K$ such that $\max_{x\in\R^n}|K\cap(x+H^\perp)|=|K\cap H^\perp|$ we have that
\begin{equation}\label{eq:dualRestrictedAAGJV}
|K\cap H^\perp||K|\geqslant{{d}\choose{d_1}}^{-1}|K\cap H_1^\perp||K\cap H_2^\perp|,
\end{equation}
where $H_j=\textrm{span}\{e_k\,:\,k\in S_j\}$ and $H=\textrm{span}\{e_k\,:\,k\in S\}$. Also a functional version of this inequality was obtained for log-concave functions.

In this manuscript we consider the more general situation, in the spirit of the extension \eqref{eq:LocalLoomisWhitney2NonOrthogonalVectors} of inequality \eqref{eq:FirstLocalLoomisWhitney}, in which we fix a basis in $\R^n$ which is not necessarily orthonormal.

Let us point out the fact that $(S_1,\dots, S_m)$ forms a uniform cover of $[n]$ with weights $(p_1,\dots p_m)$ if and only if $(S_1^c,\dots, S_m^c)$ forms a uniform cover of $[n]$ with weights $(p_1^\prime,\dots,p_m^\prime)$, where $p_j^\prime=\frac{p_j}{p-1}$ and $p=\sum_{j=1}^m p_j$. In addition,  if $\{e_i\}_{i=1}^n$ is an orthonormal basis of $\R^n$, $H_j=\textrm{span}\{e_k\,:\,k\in S_j\}$ and $\widetilde{H}_j=\textrm{span}\{e_k\,:\,k\not\in S_j\}$, then $H_j^\perp=\widetilde{H}_j$ and hence the projections onto the subspaces which are orthogonal to the ones generated by the vectors given by the uniform cover are simply the projections onto the subspaces generated by the vectors given by the uniform cover $(S_1^c,\dots, S_m^c)$  with weights $(p_1^\prime,\dots,p_m^\prime)$. However, if $\{w_i\}_{i=1}^n$ is not an orthonormal basis, it is not generally the case that the orthogonal subspace to $H_j=\textrm{span}\{w_k\,:\,k\in S_j\}$ is $\widetilde{H}_j=\textrm{span}\{w_k\,:\,k\not\in S_j\}$. Therefore, different extensions of the inequalities can be considered.

Our starting point is the following functional inequality due to Finner \cite{finner}, which recovers the Bollob\'as and Thomason inequality.
%\begin{thm}[Finner's inequality]\label{thm:invariant finner's ineq}
%	Let, for $1\leqslant j\leqslant m$, $S_j\subseteq[n]$ and $p_j>0$ be such that $(S_1^c,\dots,S_m^c)$ form a uniform cover of $[n]$ with weights $(p_1,\dots,p_m)$ and let $H_j=\textrm{span}\{e_k\,:\,k\in S_j\}$. Then, for every %integrable functions $f_j:H_j^\perp\to[0,\infty)$ we have
%$$
%\int_{\R^n}\prod_{j=1}^{m}f_j^{p_j}(P_{H_j^{\perp}} x)dx \leqslant \prod_{j=1}^m\left(\int_{H_j^{\perp}}f_j\right)^{p_j}.
%$$
%\end{thm}
%\textcolor{red}{DA: I think that we should state this in the following equivalent form, which fits more the notation we are using along the whole text:}
\begin{thm}[Finner's inequality]\label{thm:invariant finner's ineq}
	Let $(S_1,\dots, S_m)$ be a uniform cover of $[n]$ with weights $(p_1,\dots,p_m)$ and let $H_j=\textrm{span}\{e_k\,:\,k\in S_j\}$. Then, for all integrable functions $f_j:H_j\to[0,\infty)$ we have
	\begin{equation}\label{eq:ClassicalFinner}
	\int_{\R^n}\prod_{j=1}^{m}f_j^{p_j}(P_{H_j} x)dx \leqslant \prod_{j=1}^m\left(\int_{H_j}f_j\right)^{p_j}.
	\end{equation}
\end{thm}
In the framework of subspaces generated by vectors  given by a uniform  cover with weights and a (not necessarily orthonormal) basis of $\R^n$, we obtain the following affine-invariant Finner inequality.

\begin{thm}\label{thm:affineFinnerIntro}
 Fix a basis $\{w_i\}_{i=1}^n$ of $\R^n$ and let $(S_1,\dots, S_m)$ form a uniform cover of $[n]$ with weights $(p_1,\dots, p_m)$. If $\widetilde{H}_j={\rm span}\{w_k\,:\,k\not\in S_j\}$ and $p=\sum_{j=1}^mp_j$, then for any integrable $f_j:\widetilde{H}_j^\perp\to[0,\infty)$ we have
	\begin{equation}\label{eq:affineFinner}
	\int_{\R^n}\prod_{j=1}^mf_j^{p_j}(P_{\widetilde{H}_j^\perp}x)dx\leqslant BL_1((w_i)_{i=1}^n,[n],(S_j)_{j=1}^m,(p_j)_{j=1}^m)\prod_{j=1}^m\left(\int_{\widetilde{H}_j^\perp}f_j(x)dx\right)^{p_j},
	\end{equation}
	where, for any  basis $\{w_i\}_{i=1}^n$ of $\R^n$, $S\subseteq[n]$ and any uniform cover $(S_1,\dots, S_m)$ of $S$ with weights $(p_1,\dots, p_m)$, we denote
	\begin{equation}
	BL_1((w_i)_{i=1}^n,S,(S_j)_{j=1}^m,(p_j)_{j=1}^m):=\frac{\prod_{j=1}^m|\wedge_{k\in S\setminus S_j}w_k|^{p_j}}{|\wedge_{i\in S}w_i|^{p-1}}.
	\end{equation}
Here, for any vectors $w_1,\dots,w_k$, $|\wedge_{i=1}^k w_i|$ denotes the volume of the $k$-dimensional parallelepiped spanned by the vectors $w_i$.\\
Moreover, for every integrable $f\colon \R^n\to[0,\infty)$ and $f_j:\widetilde{H}^{\perp}_j\to[0,\infty)$ such that $f(x)\geqslant\prod_{j=1}^mf_j^{p_j}(x_j)$ whenever  $x=\sum_{j=1}^m p_jx_j$ for some $x_j\in \widetilde{H}_j^{\perp}$, we have
$$
\int_{\R^n}f(x)dx\geqslant\frac{1}{ BL_1([n],(S_j)_{j=1}^m,(p_j)_{j=1}^m)}\prod_{j=1}^m\left(\int_{\widetilde{H}_j^{\perp}}f_j(x)dx\right)^{p_j}.
$$
\end{thm}

\begin{rmk}
Notice that if $(w_i)_{i=1}^n=(e_i)_{i=1}^n$ is the canonical basis then $\widetilde{H}_j^\perp=H_j=\textrm{span}\{e_k\,:\,k\in S_j\}$ and this inequality becomes \eqref{eq:ClassicalFinner}. Additionally, once we have proved the first inequality, then we can use the fact that any Brascamp-Lieb inequality gives a reverse Brascamp-Lieb inequality with inverse constant (see \cite{Bar}).
\end{rmk}

The following theorem, which is an inequality when projecting on different subspaces from the ones considered in the affine Finner inequality, is an equivalent version of Theorem \ref{thm:affineFinnerIntro}.

\begin{thm}\label{thm:BrascampLiebSpannedInSjIntro}
Let $\{w_i\}_{i=1}^n$ be a basis of $\R^n$, let $m\geqslant 1$ and let $(S_1,\dots, S_m)$ be a uniform cover of $[n]$ with weights $(p_1,\dots, p_m)$. Let $H_j={\rm span}\{w_k\,:\,k\in S_j\}$. Then, for every integrable $f_j:H_j\to[0,\infty)$, $1\leqslant j\leqslant m$, we have
$$
\int_{\R^n}\prod_{j=1}^m f_j^{p_j}(P_{H_j}x)dx\leqslant BL_2((w_i)_{i=1}^n,[n],(S_j)_{j=1}^m,(p_j)_{j=1}^m)\prod_{j=1}^m\left(\int_{H_j}f_j(x)dx\right)^{p_j},
$$
where, for any basis $\{w_i\}_{i=1}^n$ of $\R^n$, $S\subseteq[n]$ and any uniform cover $(S_1,\dots, S_m)$ of $S$ with weights $(p_1,\dots,p_m)$, we denote
\begin{equation}
BL_2((w_i)_{i=1}^n, S,(S_j)_{j=1}^m,(p_j)_{j=1}^m):=\frac{\prod_{j=1}^m|\wedge_{k\in S_j}w_k|^{p_j}}{|\wedge_{i\in S} w_i|}.
\end{equation}
Moreover, for every integrable $f\colon \R^n\to[0,\infty)$ and $f_j:H_j\to[0,\infty)$ such that $f(x)\geqslant\prod_{j=1}^mf_j^{p_j}(x_j)$ whenever  $x=\sum_{j=1}^m p_jx_j$ for some $x_j\in H_j$, we have
$$
\int_{\R^n}f(x)dx\geqslant\frac{1}{ BL_2([n],(S_j)_{j=1}^m,(p_j)_{j=1}^m)}\prod_{j=1}^m\left(\int_{H_j}f_j(x)dx\right)^{p_j}.
$$
\end{thm}

\begin{rmk}
Notice that if we call $M$ the matrix whose columns are the vectors $(w_i)_{i=1}^n$, and we denote by $v_i$ the rows of the matrix $M^{-1}$,
then we have that for every $1\leqslant j\leqslant m$
$$
H_j=\textrm{span}\{w_k\,:\,k\in S_j\}=\textrm{span}\{v_k\,:\,k\not\in S_j\}^\perp.
$$
Therefore, a Brascamp-Lieb inequality projecting on the subspaces $H_j$ can be obtained from Theorem \ref{thm:affineFinnerIntro} with constant $BL_1((v_i)_{i=1}^n,[n],(S_j)_{j=1}^m,(p_j)_{j=1}^m)$. Since it turns out (see Lemma \ref{lem:EqualConstants} below) that
$$
BL_1((v_i)_{i=1}^n,[n],(S_j)_{j=1}^m,(p_j)_{j=1}^m)=BL_2((w_i)_{i=1}^n,[n],(S_j)_{j=1}^m,(p_j)_{j=1}^m),
$$
we have that Theorem \ref{thm:BrascampLiebSpannedInSjIntro} and Theorem \ref{thm:affineFinnerIntro} are equivalent. From now on we will omit the arguments in $BL_1$ and $BL_2$ whenever they are clear from the context.
\end{rmk}
In order to prove Theorem \ref{thm:affineFinnerIntro} we will use the so-called factorisation method, which was introduced and developed in \cite{CHV}. The idea of the factorisation method is that to prove our inequality we first test on a function, call it $M$, then we \textit{factorise} $M$ appropriately as a product of functions to use H\"older's inequality. In \cite{CHV} it was proved (under mild hypotheses) that a positive multilinear inequality holds if and only if such a factorisation exists.

In the Appendix we provide a proof for Theorem \ref{thm:BrascampLiebSpannedInSjIntro}, which is different from the one described above. It seems that each proof is tailored for the construction of appropriate families of subspaces.

It is interesting that we can find an explicit formula for the constant in Theorem \ref{thm:BrascampLiebSpannedInSjIntro} and Theorem \ref{thm:affineFinnerIntro}, since in general is difficult to compute the Brascamp-Lieb constant. 
%It is not clear, a priori, whether Theorem \ref{thm:BrascampLiebSpannedInSjIntro} and Theorem %\ref{thm:affineFinnerIntro} are equivalent, since the value of the constant would come in terms of the %vectors $v_i$ and not in terms of the vectors $w_i$. However, we will prove that for any basis %$\{w_i\}_{i=1}^n$ of $\R^n$, $(S_1,\dots, S_m)$ uniform cover of $[n]$ with weights $(p_1,\dots,p_m)$, then
%$$
%BL_1((v_i)_{i=1}^n,[n],(S_j)_{j=1}^m,,(p_j)_{j=1}^m)=BL_2((w_i)_{i=1}^n,[n],(S_j)_{j=1}^m,,(p_j)_{j=1}^m),
%$$
%which will show that Theorems \ref{thm:affineFinnerIntro} and \ref{thm:BrascampLiebSpannedInSjIntro} are %actually equivalent. Nevertheless, we will also present a direct proof of Theorem %\ref{thm:BrascampLiebSpannedInSjIntro} in Section \ref{sec:Brascamp-LiebInequalities}.
These inequalities will provide different affine invariant versions of Bollob\'as-Thomason inequality and its dual. In particular we will prove the following different extensions of Bollob\'as-Thomason inequality:
\begin{thm}\label{thm:AffineInvariantBollobasThomasonIntro}
Let $\{w_i\}_{i=1}^n$ be a basis of $\R^n$ and let $(S_1,\dots, S_m)$ be a uniform cover of $[n]$ with weights $(p_1,\dots p_m)$. If $H_j={\rm span}\{w_k\,:\,k\in S_j\}$, $\widetilde{H}_j={\rm span}\{w_k\,:\,k\not\in S_j\}$ and $p=\sum_{j=1}^mp_j$, then, for every compact $K\subseteq\R^n$ we have the following four inequalities:
\begin{align}
\label{eq1AffineInvariantBT}|K|&\leqslant BL_1\prod_{j=1}^m|P_{\widetilde{H}_j^\perp} K|^{p_j},\\
\label{eq2AffineInvariantBT}|K|^{p-1}&\leqslant  BL_2\prod_{j=1}^m|P_{H_j^\perp} K|^{p_j},\\
\label{eq3AffineInvariantBT}|K|&\leqslant  BL_2\prod_{j=1}^m|P_{H_j} K|^{p_j},\\
\label{eq4AffineInvariantBT}|K|^{p-1}&\leqslant  BL_1\prod_{j=1}^m|P_{\widetilde{H}_j} K|^{p_j}.
\end{align}
\end{thm}

Using \eqref{eq1AffineInvariantBT} in Theorem \ref{thm:AffineInvariantBollobasThomasonIntro} we obtain a generalisation of Gagliardo-Nirenberg inequality. Moreover, for log-concave functions we can achieve a better constant which also leads to a sharp generalisation of an inequality obtained by Bobkov and Nazarov in \cite{BN}, implying the boundedness of the isotropic constant of log-concave unconditional measures.

In general, we will make an extensive study of the different affine invariant extensions of Bollob\'as-Thomason inequality that we obtain. We will also provide functional versions of the geometric inequalities obtained in the context of log-concave functions, which will recover their geometric versions when applied to functions of the form $e^{-\Vert\cdot\Vert_K}$, where $\Vert\cdot\Vert_K$ denotes the Minkowski functional associated to a given convex body $K$ containing the origin in its interior. These will be different from the ones we get when considering the characteristic function $\chi_K$, as is the case with Brascamp-Lieb inequality. These functional inequalities cannot be directly obtained from the stated Brascamp-Lieb inequality as in this setting one would have to consider an extra dimension and the subspaces on which one would have to project would not form a covering of $[n+1]$. Rather than that, they will be restricted Loomis-Whitney type inequalities in $\R^{n+1}$. We will also prove affine-invariant local versions of both geometric and functional inequalities. In particular, we will prove the following different extensions of inequality \eqref{eq:LocalLoomisWhitney2NonOrthogonalVectors}.

\begin{thm}\label{thm:AffineGeometricLocalLoomisWhitneyIntro}
Let $\{w_i\}_{i=1}^n$ be a basis of $\R^n$ and let $S\subseteq[n]$ with cardinality $|S|=d$. Let $(S_1,\dots, S_m)$ form a uniform cover of $S$ with weights $(p_1,\dots, p_m)$. If $H={\rm span}\{w_k\,:\,k\in S\}$, $H_j={\rm span}\{w_k\,:\,k\in S_j\}$, $\widetilde{H}_j={\rm span}\{w_k\,:\,k\in S\setminus S_j\}$, $d_j={\rm dim}H_j=|S_j|$, and $p=\sum_{j=1}^mp_j$. Then, for every convex body $K\subseteq\R^n$ we have the following four inequalities:
\begin{align}
\label{eq:1.11}|P_{H^\perp}K|^{p-1}|K|&\leqslant BL_1\cdot\frac{\prod_{j=1}^m{{n-d+d_j}\choose{d_j}}^{p_j}}{{n\choose d}}\prod_{j=1}^m|P_{\widetilde{H}_j^\perp}K|^{p_j},\\
\label{eq:1.12}|P_{H^\perp}K||K|^{p-1}&\leqslant BL_2\cdot\frac{\prod_{j=1}^m{{n-d_j}\choose{n-d}}^{p_j}}{{n\choose d}^{p-1}}\prod_{j=1}^m|P_{H_j^\perp}K|^{p_j},\\
\label{eq:1.13}|P_{H^\perp}K|^{p-1}|K|&\leqslant BL_2\cdot\frac{\prod_{j=1}^m{{n-d+d_j}\choose{d_j}}^{p_j}}{{n\choose d}}\prod_{j=1}^m|P_{H_j\oplus H^\perp}K|^{p_j},\\
\label{eq:1.14}|P_{H^\perp}K||K|^{p-1}&\leqslant BL_1\cdot\frac{\prod_{j=1}^m{{n-d_j}\choose{n-d}}^{p_j}}{{n\choose d}^{p-1}}\prod_{j=1}^m|P_{\widetilde{H}_j\oplus H^\perp}K|^{p_j}.
\end{align}
\end{thm}

\begin{rmk}
Notice that if $S=\{1,2\}\subseteq[n]$, $S_1=\{1\}$, and $S_2=\{2\}$, then necessarily $p_1=p_2=1$. Taking any linearly independent $w_1, w_2\in S^{n-1}$, $H={\rm span}\{w_1,w_2\}$ and an orthonormal basis $\{w_3,\dots ,w_n\}$ of $H^\perp$, we have that $w_1\wedge w_2=\sqrt{1-\langle w_1,w_2\rangle^2}$ and if $v_1\in S^{n-1}$ spans $w_1^\perp\cap H$ and $v_2\in S^{n-1}$ spans $w_2^\perp\cap H$, then also $v_1\wedge v_2=w_1\wedge w_2$ and applying either \eqref{eq:1.11} or \eqref{eq:1.12} to $\{w_1,w_2,w_3,\dots,w_n\}$, or \eqref{eq:1.13} or \eqref{eq:1.14} to $\{v_1,v_2,w_3\dots,w_n\}$, any of the previous four inequalities recovers inequality \eqref{eq:LocalLoomisWhitney2NonOrthogonalVectors}.
\end{rmk}

Regarding dual Loomis-Whitney type inequalities, we will prove the following different extensions of inequality \eqref{eq:LiakopoulosDual}.

\begin{thm}\label{thm:DualBollobasThomasonIntro}
Let $w_1,\dots,w_n$ be $n$ vectors spanning $\R^n$, let $m\geqslant 1$ and let $(S_1,\dots, S_m)$ be a uniform cover of $[n]$ with weights $(p_1,\dots, p_m)$. Let $H_j={\rm span}\{w_k\,:\,k\in S_j\}$, $\widetilde{H}_j={\rm span}\{w_k\,:\,k\not\in S_j\}$, $d_j={\rm dim}H_j=|S_j|$, $\widetilde{d}_j={\rm dim}\widetilde{H}_j=n-d_j$, and $p=\sum_{j=1}^mp_j$. Then, for every convex body $K\subseteq\R^n$ containing the origin we have the following inequalities:
\begin{align}
\label{eq1:DualBollobasIntro}|K|&\geqslant\frac{1}{BL_1}\cdot\frac{\prod_{j=1}^m(d_j!)^{p_j}}{n!}\prod_{j=1}^m|K\cap\widetilde{H}_j^\perp|^{p_j},\\
\label{eq2:DualBollobasIntro}|K|^{p-1}&\geqslant\frac{1}{BL_2}\cdot\frac{(\widetilde{d}_j!)^{p_j}}{(n!)^{p-1}}\prod_{j=1}^m|K\cap H_j^\perp|^{p_j},
\\
\label{eq3:DualBollobasIntro}|K|&\geqslant\frac{1}{BL_2}\frac{\prod_{j=1}^m(d_j!)^{p_j}}{n!}\prod_{j=1}^m|K\cap H_j|^{p_j},\\
\label{eq4:DualBollobasIntro}|K|^{p-1}&\geqslant\frac{1}{BL_1}\frac{\prod_{j=1}^m(\widetilde{d}_j!)^{p_j}}{(n!)^{p-1}}\prod_{j=1}^m|K\cap\widetilde{H}_j|^{p_j}.
\end{align}
\end{thm}

Some restricted versions will be proved too. For instance, we will prove the following
\begin{thm}\label{thm:RestrictedDualBollobasThomason1Intro}
Let $\{w_i\}_{i=1}^n$ be a basis of $\R^n$ and let $S\subseteq[n]$ with cardinality $|S|=d$. Let $(S_1,\dots, S_m)$ form a uniform cover of $S$ with weights $(p_1,\dots, p_m)$. Set $H={\rm span}\{w_k\,:\,k\in S\}$, $H_j={\rm span}\{w_k\,:\,k\in S_j\}$, $\widetilde{H}_j={\rm span}\{w_k\,:\,k\in S\setminus S_j\}$, $d_j={\rm dim}H_j=|S_j|$, $\widetilde{d}_j={\rm dim}\widetilde{H}_j=d-d_j$ and $p=\sum_{j=1}^mp_j$. Then, for every convex body $K\subseteq\R^n$, we have the following four inequalities:
\begin{align}
\max_{x\in H}|K\cap(x+H^\perp)|^{p-1}|K|&\geqslant\frac{1}{BL_1}\cdot\frac{\prod_{j=1}^m(d_j)^{p_jd_j}}{d^d}\prod_{j=1}^m|K\cap\widetilde{H}_j^\perp|^{p_j},\\
\max_{x\in H}|K\cap(x+H^\perp)||K|^{p-1}&\geqslant\frac{1}{BL_2}\cdot\frac{\prod_{j=1}^m(\widetilde{d}_j)^{p_j(\widetilde{d}_j)}}{d^{d(p-1)}}\prod_{j=1}^m|K\cap H_j^\perp|^{p_j},\\
\max_{x\in H}|K\cap(x+H^\perp)|^{p-1}|K|&\geqslant\frac{1}{BL_2}\cdot\frac{\prod_{j=1}^m(d_j)^{p_jd_j}}{d^d}\prod_{j=1}^m|K\cap (H_j\oplus H^\perp)|^{p_j},\\
\max_{x\in H}|K\cap(x+H^\perp)||K|^{p-1}&\geqslant\frac{1}{BL_1}\cdot\frac{\prod_{j=1}^m(\widetilde{d}_j)^{p_j\widetilde{d}_j}}{d^{d(p-1)}}\prod_{j=1}^m|K\cap(\widetilde{H}_j\oplus H^\perp)|^{p_j}.
\end{align}
\end{thm}

\begin{rmk}
Notice that no assumption on the barycentre was made. Taking into account (see \cite{F}) that if $K$ is a centred convex body and $H\in G_{n,d}$ then
$$
\max_{x\in H}|K\cap(x+H^\perp)|\leqslant \left(\frac{n+1}{n-d+1}\right)^{n-d}|K\cap H^\perp|
$$
we can obtain estimates in terms of the volume of sections through the centroid. If we assume that the section of maximal volume with subspaces parallel to $H^\perp$ is the one through the origin, then the value of the constant in the inequalities in Theorem \ref{thm:RestrictedDualBollobasThomason1Intro} is better in some special cases, providing extensions of inequality \eqref{eq:dualRestrictedAAGJV}.
\end{rmk}

The paper is structured as follows. In Section \ref{sec:Preliminaries} we introduce the notation that we use, as well as provide the necessary known results that we use in the paper. In Section \ref{sec:Brascamp-LiebInequalities} we collect various versions of the affine invariant Finner inequality and its reverse, as well as the proof of Theorem \ref{thm:affineFinnerIntro}, which we will use in order to prove several extensions of the aforementioned inequalities. Sections \ref{sec:Loomis-WhitneyInequalities} and \ref{sec:DualLoomis-WhiteneyInequalities} will be devoted to the proof of the Loomis-Whitney and dual Loomis-Whitney inequalities, respectively, as well as their functional and restricted versions. %Finally, in Section \ref{sec:ReverseLocalLoomisWhitney}, we will prove a restricted version of inequality \eqref{eq:ReverseLoomisWhitneyCGG} and give estimates for the constant $\gamma$ in inequality \eqref{eq:ReverseDualLoomisWhitney} as well as for the constant in some restricted versions of it.

\section{Preliminaries}\label{sec:Preliminaries}
In this section we provide the necessary background that we will use in order to prove the different versions of Loomis-Whitney and dual Loomis-Whitney type inequalities.

\subsection{Log-concave functions}

A function $f\colon \R^n\to[0,\infty)$ is called log-concave if $f(x)=e^{-v(x)}$ where $v:\R^n\to(-\infty,\infty]$ is a convex function. It is well-known that any integrable log-concave function $f\colon \R^n\to[0,\infty)$ is bounded and has moments of all orders (see, for instance \cite[Lemma 2.2.1]{BGVV}). If $K\subseteq\R^n$ is a convex body then its characteristic function $\chi_K$ is integrable and log-concave, with integral $|K|$ and if $K$ is a convex body containing the origin, then $e^{-\Vert \cdot\Vert_K}$, where $\Vert x\Vert_K=\inf\{\lambda>0\,:\,x\in\lambda K\}$ is the Minkowski functional associated to $K$, is integrable and log-concave, with integral $n!|K|$. The set of integrable log-concave functions in $\R^n$ will be denoted by $\mathcal{F}(\R^n)$.

For any $f:\R^n\to\R$ and any $H\in G_{n,k}$, the projection of $f$ onto $H$ is the function defined by
$$
P_Hf(x)=\sup_{y\in H^\perp}f(x+y)\quad x\in H.
$$
Notice that $\Vert P_Hf\Vert_\infty=\Vert f\Vert_\infty$ and that if $f\in\mathcal{F}(\R^n)$ $f(x)=\Vert f\Vert_\infty e^{-v(x)}$ where $v:\R^n\to[0,\infty)$ is a convex function then $P_Hf$ is the log-concave function, integrable on $H$, $P_Hf(x)=\Vert f\Vert_\infty e^{-w(x)}$ where $w:H\to[0,\infty)$ is the convex function whose epigraph, ${\rm epi}(w)=\{(x,t)\in H\times[0,\infty)\,:\,w(x)\geqslant t$\}, is the projection onto $\overline{H}={\rm span}\{H, e_{n+1}\}$ of ${\rm epi}(v)=\{(x,t)\in\R^n\times[0,\infty)\,:\,v(x)\geqslant t\}$. If $K\subseteq\R^n$ is a convex body, then for any $H\in G_{n,k}$ we have that $P_H\chi_K=\chi_{P_HK}$ and if in addition $K$ contains the origin then $P_H e^{-\Vert\cdot\Vert_K}=e^{-\Vert\cdot\Vert_{P_HK}}$. Given any integrable log-concave function $f$ on a linear subspace $H\in G_{n,k}$ we will denote by $\Vert f\Vert_1$ its integral on the subspace $H$.

\subsection{Berwald's inequality}
Berwald's inequality, which was proved in \cite{Ber} for $0<\gamma_1<\gamma_2$ and was extended to the range $-1<\gamma_1<\gamma_2$ in \cite[Theorem 5.1]{GZ}, states the following.

\begin{thm} \label{berwald01}
Let $-1<\gamma_1<\gamma_2$, $K\subseteq\R^n$ a convex body, and $f\colon K\to[0,+\infty)$ be concave, continuous, and not identically zero. Then,
$$
\left(\frac{{\gamma_2+n\choose n}}{|K|}\int_K f(x)^{\gamma_2}\,dx\right)^{1/\gamma_2}
\leqslant
\left(\frac{{\gamma_1+n\choose n}}{|K|}\int_K f(x)^{\gamma_1}\,dx\right)^{1/\gamma_1}.
$$
Equality holds if and only if $f$ is a roof function over a point in $K$.
\end{thm}

%\begin{thm}\label{berwald}
%Let $K\in\mathcal K^n$, $f_1,\dots f_m:K\to[0,+\infty)$ concave, continuous, and non identically null functions, and $\alpha_1,\dots, \alpha_m>0$. Then
%$$
%\frac{1}{|K|}\int_K \prod_{i=1}^m f_i(x)^{\alpha_i}\,dx
%\le
%\frac{{\alpha_1+n\choose n}\cdots {\alpha_m+n\choose n}}
%{{\alpha_1+\cdots +\alpha_m+n\choose n}}
%\prod_{i=1}^m \frac{1}{|K|}\int_Kf_k(x)^{\alpha_k}\,dx.
%$$
%Equality holds for $m>1$ if and only if all the $f_i$'s are roof functions over the same point in $K$.
%\end{thm}

Functional versions of the above for functions in $\mathcal{F}(\R^n)$, the set of integrable log-concave functions on $\R^n$, was proved in  \cite[Lemma 3.3]{AAGJV} for the range $\gamma>0$, and  was extended to the range $\gamma>-1$ in \cite[Theorem 1.1]{ABG}.

\begin{lemma}\label{lem:BerwaldExponentialDensity}
Let $f\in\mathcal{F}(\R^n)$ and let $C$ be the convex set $C=\{(x,t)\in\R^n\times [0,\infty):f(x)\geqslant e^{-t}\Vert f\Vert_\infty\}$. Let $h:C\to[0,\infty)$ be a continuous, concave, not identically zero function. Then,
$$
\Phi_\gamma(h):=\left(\frac{1}{\Gamma\left(1+\gamma\right)\int_C e^{-t}dxdt}\int_L h^\gamma(x,t)e^{-t}dxdt\right)^\frac{1}{\gamma}
$$
is decreasing in $\gamma\in(-1,\infty)$.
\end{lemma}

%\begin{thm}\label{thm:BerwaldExponential}
%Let $f\in\mathcal{F}(\R^n)$ and let $L$ be the convex set $L=\{(x,t)\in\R^n\times [0,\infty):f(x)\geqslant e^{-t}\Vert f\Vert_\infty\}$. Let $h_1,\dots,h_m:L\to[0,\infty)$ be continuous, concave, not identically null functions, %$\alpha_1,\dots,\alpha_m>0$ and $\sigma=\alpha_1+\dots+\alpha_m$. Then,
%$$
%\frac{1}{\int_L e^{-t}dxdt}\int_L \prod_{i=1}^mh_i^{\alpha_i}(x,t)e^{-t}dxdt\leqslant \frac{\Gamma\left(1+\sigma\right)}{\prod_{i=1}^m\Gamma\left(1+\alpha_i\right)}\prod_{i=1}^m \frac{1}{\int_L e^{-t}dxdt}\int_L %h_i^{\alpha_i}(x,t)e^{-t}dxdt.
%$$
%\end{thm}

The following result can be seen as a degenerate version of Lemma \ref{lem:BerwaldExponentialDensity} and its proof can be found in \cite[Lemma 2.1]{ABG}. In this case we can characterise the equality cases.

\begin{lemma}\label{lem:0-dimBerwaldExponentialDensity}
Let $h:[0,\infty)\to[0,\infty)$ be a continuous, concave, not identically zero function. Then,
$$
\Phi_\gamma(h):=\left(\frac{1}{\Gamma\left(1+\gamma\right)}\int_0^\infty h^\gamma(t)e^{-t}dt\right)^\frac{1}{\gamma}
$$
is decreasing in $\gamma\in(-1,\infty)$. Furthermore, if there exists $-1<\gamma_1<\gamma_2$ such that $\Phi_{\gamma_1}(h)=\Phi_{\gamma_2}(h)$ then $h$ is a linear function and $\Phi_\gamma$ is constant in $\gamma$.
\end{lemma}

\section{Affine Invariant Finner Inequality}\label{sec:Brascamp-LiebInequalities}

In this section we prove the affine invariant Finner inequality, already mentioned in the introduction, together with its reverse inequality.

\begin{thm}\label{thm:AffineFinner}
	Let $\{w_i\}_{i=1}^n$ be a basis of $\R^n$ and let $(S_1,\dots, S_m)$ form a uniform cover of $[n]$ with weights $(p_1,\dots p_m)$. If $\widetilde{H}_j={\rm span}\{w_k\,:\,k\not\in S_j\}$ and $p=\sum_{j=1}^mp_j$, then for any integrable $f_j:\widetilde{H}_j^\perp\to[0,\infty)$ we have
	$$
	\int_{\R^n}\prod_{j=1}^mf_j^{p_j}(P_{\widetilde{H}_j^\perp}x)dx\leqslant\frac{\prod_{j=1}^m|\wedge_{k\not\in S_j}w_k|^{p_j}}{|\wedge_{i=1}^nw_i|^{p-1}}\prod_{j=1}^m\left(\int_{\widetilde{H}_j^\perp}f_j(x)dx\right)^{p_j}.
	$$
	Moreover, for every integrable $f\colon \R^n\to[0,\infty)$ and $f_j:\widetilde{H}_j^\perp\to[0,\infty)$ such that $f(x)\geqslant\prod_{j=1}^mf_j^{p_j}(x_j)$ whenever  $x=\sum_{j=1}^m p_jx_j$ for some $x_j\in \widetilde{H}_j^\perp$, we have
	$$
	\int_{\R^n}f(x)dx\geqslant\frac{|\wedge_{i=1}^nw_i|^{p-1}}{\prod_{j=1}^m|\wedge_{k\not\in S_j}w_k|^{p_j}}\prod_{j=1}^m\left(\int_{\widetilde{H}_j^\perp}f_j(x)dx\right)^{p_j}.
	$$
\end{thm}

\begin{proof}

 Let $Q=|\wedge_{i=1}^nw_i|^{-1/n} $.
	Given $M \in L^{p'}(\R^n)$ with $\|M\|_{p'} = 1$ we write it as
	a telescoping product
	$$ M(x)^{p'} = \gamma_n(x) \dots \gamma_1(x),$$
	where for every $1\leqslant i\leqslant n$
	$$ \gamma_i(x) = Q\; \frac{\int_{\R^{n-i}} M(x+ s_{i+1} \omega_{i+1} + \dots + s_{n} \omega_{n})^{p'} {\rm d} s_n \dots {\rm d} s_{i+1}}
	{\int_{\R^{n-i+1}} M(x+ s_{i} \omega_{i} + \dots  + s_{n} \omega_{n} )^{p'} {\rm d} s_n \dots {\rm d} s_{i}}.$$
	Notice that the numerator in $\gamma_n$ is just $M(x)^{p'}$ while the denominator in $\gamma_1$ is $Q^n$. Moreover for every $1\leqslant i\leqslant n$ and every $x\in\R^n$
	$$ \int_{\R} \gamma_i(x + t \omega_i) {\rm d} t = Q.$$
	
	\medskip
	\noindent
	Let, for every $1 \leqslant j \leqslant m$,
	$$ M_j(x) = \prod_{i =1}^n \gamma_i(x)^{b_{ij}}$$
	where $b_{ij} = 1 - \chi_{S_j}(i) = \chi_{S_j^c}(i)$, and notice that
	$$ M(x) = \prod_{j=1}^m M_j(x)^{p_j/p}.$$
	
	\medskip
	\noindent
For $x\in\mathbb{R}^n$ we write $\displaystyle{x=\sum_{k \in S_j} \lambda_k \omega_k + \sum_{k \not\in S_j} s_k\omega_k}$ and we have 
%	$$ \int_{\widetilde{H}_j} M_j(u + v)dv =
%	\int  M_j(u + \sum_{i \in T_j} s_i \omega_i ) \left| \frac{\partial v}{ \partial s} \right| \prod_{i \in T_j}  {\rm d}  s_i$$
%	$$ = |\wedge_{i \in T_j} \omega_i|  \int  M_j(u + \sum_{i \in T_j} s_i \omega_i ) \prod_{i \in T_j}  {\rm d}  s_i.$$
	
\begin{equation*}\int_{\widetilde{H}_j} M_j(x) dv
= |\wedge_{k \not\in S_j} \omega_k| \int_{\R^{n-|S_j|}}\left( \sum_{k \in S_j} \lambda_k \omega_k + \sum_{k \not\in S_j} s_k\omega_k \right) \prod_{k \not\in S_j} ds_k.
\end{equation*}
	
	\medskip
	\noindent
	So we consider for every $1\leqslant j\leqslant m$
	$$ \int_{\R^{n-|S_j|}}M_j( s_1 \omega_1 + \dots +  s_n \omega_n ) \prod_{k \not\in S_j} ds_k$$
	and look for uniform bounds in $s_i$ for $i \in S_j$.
	
	\medskip
	\noindent
	Now, for each $1\leqslant j\leqslant m$
$$\int_{\R^{n-|S_j|}} M_j( s_1 \omega_1 + \dots +  s_n \omega_n ) \prod_{k\not\in S_j} ds_k$$
$$=\int_{\R^{n-|S_j|}} \prod_{i=1}^n \gamma_{i}( s_1 \omega_1 + \dots +  s_n \omega_n)^{b_{ij}} \prod_{k\not\in S_j}ds_{k}$$
$$=\int \prod_{k \not\in S_j} \gamma_{k}( s_1 \omega_1 + \dots +  s_n
	\omega_n) \prod_{k\not\in S_j} ds_{k}.$$

%	=  \int \prod_{i' \in T_j} \left[ \gamma_{i'}( s_1 \omega_1 + \dots +  s_n
%	\omega_n) {\rm d} s_{i'} \right].$$
	
	\medskip
	\noindent
	However, for every $1\leqslant i\leqslant n$, and every $(s_1,\dots,s_n)\in\R^n$
$$\gamma_i( s_1 \omega_1 + \dots +  s_n \omega_n)$$
$$= Q \; \frac{\int_{\R^{n-i}}M(s_1 \omega_1 + \dots +  s_n \omega_n + t_{i+1} \omega_{i+1} + \dots + t_{n} \omega_{n})^{p'} dt_n \dots dt_{i+1}}
	{\int_{\R^{n-i+1}} M(s_1 \omega_1 + \dots +  s_n \omega_n+  t_{i} \omega_{i} + \dots + t_n \omega_n)^{p'} dt_n \dots dt_{i}}$$
$$=Q \; \frac{\int_{\R^{n-i}} M(s_1 \omega_1 + \dots +  s_n \omega_n)^{p'} d s_n \dots d s_{i+1}}
	{\int_{\R^{n-i+1}} M(s_1 \omega_1 + \dots +  s_n \omega_n)^{p'} d s_n \dots d s_{i}}
$$
	is a function $\tilde{\gamma}_i(s_1, \dots , s_i)$ only of the variables $s_1, \dots , s_i$, and for every $(s_1,\dots s_{i-1})\in\R^{i-1}$ the integral $\displaystyle{\int_\R\tilde{\gamma}_i(s_1, \dots , s_i)ds_i}$
	is exactly $Q$.
	
	\medskip
	\noindent
	Therefore, for each $1\leqslant j\leqslant m$, if $S_j^c = \{ i_1 < i_2 < \dots < i_r\}$, with $r=|S_j^c|=n-|S_j|$, we have that
	$$\int_{\R^{n-|S_j|}} M_j( s_1 \omega_1 + \dots +  s_n \omega_n ) \prod_{k \not\in S_j} ds_{k}$$ $$  = \int_{\R^{n-|S_j|}} \tilde{\gamma}_{i_1}(s_1, \dots , s_{i_1}) \tilde{\gamma}_{i_2}(s_1, \dots , s_{i_2}) \dots \tilde{\gamma}_{i_r}(s_1, \dots , s_{i_r}) ds_{i_r} \dots
	ds_{i_2}{\rm d} s_{i_1} = Q^{|S_j^c|}.$$
	
	\medskip
	\noindent
	Therefore, for every $u\in\widetilde{H}_j^\perp$
	$$\int_{\widetilde{H}_j} M_j(u + v) dv = |\wedge_{i \not\in S_j} \omega_i|Q^{|S_j^c|}.$$

	\medskip
	\noindent
	Finally,
	$$ \int_{\mathbb{R}^n} \prod_{j=1}^m f_j(P_{\tilde{H}_j^\perp}x)^{p_j} dx = \left\Vert\prod_{j=1}^m f_j(P_{\tilde{H}_j^\perp} \cdot)^{p_j/p} \right\Vert_p^p$$
	$$  = \sup_{ \Vert M\Vert_{p'} = 1}  \left[\int_{\mathbb{R}^n} \prod_{j=1}^m f_j(P_{\tilde{H}_j^\perp} x)^{p_j/p} M(x) dx\right]^p$$
	and for a given $M\in L^{p^\prime}(\R^n)$ with $\Vert M\Vert_{p^\prime}=1$,
	$$ \int_{\mathbb{R}^n} \prod_{j=1}^m f_j(P_{\tilde{H}_j^\perp} x)^{p_j/p} M(x) dx
	= \int_{\mathbb{R}^n} \prod_{j=1}^m \left[f_j(P_{\tilde{H}_j^\perp} x)M_j(x)\right]^{p_j/p}{\rm
		d} x
	$$
	$$ \leq \prod_{j=1}^m \left(\int_{\mathbb{R}^n} f_j(P_{\tilde{H}_j^\perp} x)M_j(x){\rm d} x \right)^{p_j/p}.$$
	Now, for each $1\leqslant j\leqslant m$
	$$ \int_{\mathbb{R}^n} f_j(P_{\tilde{H}_j^\perp}x)M_j(x)dx = \int_{\tilde{H}_j^\perp} \int_{\tilde{H}_j}
	f_j(u) M_j(u +v) dv du$$
	$$ \leq  \left(\int_{\tilde{H}_j^\perp} f_j(u)du\right) \sup_{u \in \tilde{H}_j^\perp}   \int_{\tilde{H}_j} M_j(u
	+v) dv
	=  \left(\int_{\tilde{H}_j^\perp} f_j(u)du \right) |\wedge_{k \not\in S_j} \omega_k| Q^{|S_j^c|}.$$
	
	\medskip
	\noindent
	Therefore,
	$$ \int_{\mathbb{R}^n} \prod_{j=1}^m f_j(P_{\tilde{H}_j^\perp}x)^{p_j} dx \leq \prod_{j=1}^m\left( |\wedge_{k\not\in S_j} \omega_k| Q^{|S_j^c|} \right)^{p_j}
	\prod_{j=1}^m \left(\int_{\tilde{H}_j^\perp} f_j \right)^{p_j} .$$
	
	\medskip
	\noindent
	Now, because of the covering, notice that $$\sum_{j=1}^m p_j|S_j^c|=\sum_{j=1}^m p_j(n-|S_j|)=np-n,$$
	therefore
	$$ \prod_{j=1}^m Q^{p_j |S_j^c|} = Q^{n(p-1)} = |\wedge_{i=1}^nw_i|^{-(p-1)}.$$
	
	\medskip
	\noindent
	We conclude that
	$$ \int_{\mathbb{R}^n} \prod_{j=1}^m f_j(P_{\tilde{H}_j^\perp}x)^{p_j} d x \leq \frac{\prod_{j=1}^m | \wedge_{i \not\in S_j} \omega_i|^{p_j}}{|\wedge_{i=1}^nw_i|^{p-1}} \prod_{j=1}^m \left(\int_{\tilde{H}_j^\perp} f_j\right)^{p_j}$$ as stated.
	
		Since the constant in the reverse Brascamp-Lieb inequality is the inverse of the constant in Brascamp-Lieb inequality, see \cite{Bar}, we also obtain the reverse inequality.
\end{proof}

\noindent

Taking into account that $(S_1,\dots, S_m)$ forms a uniform cover of $[n]$ with weights $(p_1,\dots, p_m)$ if and only if $(S_1^c,\dots, S_m^c)$ forms a uniform cover of $[n]$ with weights $(p_1^\prime,\dots,p_m^\prime)$ we have the following equivalent theorem.

\begin{thm}\label{thm:AffineFinner2}
Let $\{w_i\}_{i=1}^n$ be a basis of $\R^n$ and let $(S_1,\dots, S_m)$ form a uniform cover of $[n]$ with weights $(p_1,\dots p_m)$. If $H_j={\rm span}\{w_k\,:\,k\in S_j\}$ and $p=\sum_{j=1}^mp_j$, then for any integrable $f_j:H_j^\perp\to[0,\infty)$ we have
$$
\int_{\R^n}\prod_{j=1}^mf_j^\frac{p_j}{p-1}(P_{H_j^\perp}x)dx\leqslant\frac{\prod_{j=1}^m|\wedge_{k\in S_j}w_k|^\frac{p_j}{p-1}}{|\wedge_{i=1}^nw_i|^\frac{1}{p-1}}\prod_{j=1}^m\left(\int_{H_j^\perp}f_j(x)dx\right)^\frac{p_j}{p-1}.
$$
Moreover, for any integrable $f\colon \R^n\to[0,\infty)$ and $f_j:H_j^\perp\to[0,\infty)$ such that $f(x)\geqslant\prod_{j=1}^mf_j^\frac{p_j}{p-1}(x_j)$ whenever  $x=\sum_{j=1}^m \frac{p_j}{p-1}x_j$ for some $x_j\in H_j^\perp$, we have
$$
\int_{\R^n}f(x)dx\geqslant\frac{|\wedge_{i=1}^nw_i|^\frac{1}{p-1}}{\prod_{j=1}^m|\wedge_{k\in S_j}w_k|^\frac{p_j}{p-1}}\prod_{j=1}^m\left(\int_{H_j^\perp}f_j(x)dx\right)^\frac{p_j}{p-1}.
$$
\end{thm}

We continue with the different version of the affine-invariant Finner inequality that we also presented in the introduction. The difference is in the way that we define the subspaces.
\begin{thm}\label{thm:BrascampLiebSpannedInSj}
Let $w_1,\dots,w_n$ be $n$ vectors spanning $\R^n$, let $m\geqslant 1$ and let $(S_1,\dots, S_m)$ be a uniform cover of $[n]$ with weights $(p_1,\dots, p_m)$. Let $H_j={\rm span}\{w_k\,:\,k\in S_j\}$. Then, for any integrable $f_j:H_j\to[0,\infty)$, $1\leqslant j\leqslant m$ we have
$$
\int_{\R^n}\prod_{j=1}^m f_j^{p_j}(P_{H_j}x)dx\leqslant\frac{\prod_{j=1}^m|\wedge_{k\in S_j}w_k|^{p_j}}{|\wedge_{i=1}^n w_i|}\prod_{j=1}^m\left(\int_{H_j}f_j(x)dx\right)^{p_j}.
$$
Moreover, for any integrable $f\colon \R^n\to[0,\infty)$ and $f_j:H_j\to[0,\infty)$ such that $f(x)\geqslant\prod_{j=1}^mf_j^{p_j}(x_j)$ whenever  $x=\sum_{j=1}^m p_jx_j$ for some $x_j\in H_j$, we have
$$
\int_{\R^n}f(x)dx\geqslant\frac{|\wedge_{i=1}^n w_i|}{\prod_{j=1}^m|\wedge_{k\in S_j}w_k|^{p_j}}\prod_{j=1}^m\left(\int_{H_j}f_j(x)dx\right)^{p_j}.
$$
\end{thm}
As we already mentioned, the above theorem is equivalent to Theorem \ref{thm:AffineFinner} as we will see in the end of this section. However, in the appendix we will provide a direct proof of Theorem \ref{thm:BrascampLiebSpannedInSj}, which is tailored for this result.

As before, taking into account that $(S_1,\dots, S_m)$ forms a uniform cover of $[n]$ with weights $(p_1,\dots, p_m)$ if and only if $(S_1^c,\dots, S_m^c)$ forms a uniform cover of $[n]$ with weights $(p_1^\prime,\dots,p_m^\prime)$ we have the following equivalent theorem.
\begin{thm}\label{thm:BrascampLiebSpannedInSjc}
Let $w_1,\dots,w_n$ be $n$ vectors spanning $\R^n$, let $m\geqslant 1$ and let $(S_1,\dots, S_m)$ be a uniform cover of $[n]$ with weights $(p_1,\dots, p_m)$. Let $\widetilde{H}_j={\rm span}\{w_k\,:\,k\not\in S_j\}$ and $p=\sum_{j=1}^mp_j$. Then, for any integrable $f_j:\widetilde{H}_j\to[0,\infty)$, $1\leqslant j\leqslant m$ we have
$$
\int_{\R^n}\prod_{j=1}^m f_j^\frac{p_j}{p-1}(P_{\widetilde{H}_j}x)dx\leqslant\frac{\prod_{j=1}^m|\wedge_{k\not\in S_j}w_k|^\frac{p_j}{p-1}}{|\wedge_{i=1}^n w_i|}\prod_{j=1}^m\left(\int_{\widetilde{H}_j}f_j(x)dx\right)^\frac{p_j}{p-1}.
$$
Moreover, for any integrable $f\colon \R^n\to[0,\infty)$ and $f_j\colon\widetilde{H}_j\to[0,\infty)$ such that $f(x)\geqslant\prod_{j=1}^mf_j^\frac{p_j}{p-1}(x_j)$ whenever $x=\sum_{j=1}^m \frac{p_j}{p-1}x_j$ for some $x_j\in \widetilde{H}_j$, we have
$$
\int_{\R^n}f(x)dx\geqslant\frac{|\wedge_{i=1}^n w_i|}{\prod_{j=1}^m|\wedge_{k\in S_j}w_k|^\frac{p_j}{p-1}}\prod_{j=1}^m\left(\int_{\widetilde{H}_j}f_j(x)dx\right)^\frac{p_j}{p-1}.
$$
\end{thm}

Finally, let us show that Theorems \ref{thm:affineFinnerIntro} and \ref{thm:BrascampLiebSpannedInSjIntro} are equivalent. It is a consequence of the following:

\begin{lemma}\label{lem:EqualConstants}
Let $\{w_i\}_{i=1}^n$ be a basis of $\R^n$ and let $M$ be the matrix whose columns are the vectors $w_i$. Let $\{v_i\}_{i=1}^n$ be the basis of $\R^n$ given by the rows of the matrix $M^{-1}$. Let $m\geqslant 1$ and $(S_1,\dots, S_m)$ be a uniform cover of $[n]$ with weights $(p_1,\dots, p_m)$. Then
$$
BL_1((v_i)_{i=1}^n,[n],(S_j)_{j=1}^m,,(p_j)_{j=1}^m)=BL_2((w_i)_{i=1}^n,[n],(S_j)_{j=1}^m,,(p_j)_{j=1}^m),
$$
\end{lemma}

\begin{proof}
We have to prove that
$$
\frac{\prod_{j=1}^m|\wedge_{k\not\in S_j} v_k|^{p_j}}{|\wedge_{i=1}^n v_i|^{p-1}}=\frac{\prod_{j=1}^m|\wedge_{k\in S_j} w_k|^{p_j}}{|\wedge_{i=1}^n w_i|},
$$
where $p=\sum_{j=1}^n p_j$. Since $|\wedge_{i=1}^n v_i|=|\textrm{det}M|^{-1}=|\wedge_{i=1}^n w_i|^{-1}$ and $p=\sum_{j=1}^m p_j$, it is enough to see that for each $1\leqslant j\leqslant m$
\begin{equation}\label{eq:EqualConstantsEachj}
|\wedge_{k\not\in S_j} v_k|=\frac{|\wedge_{k\in S_j} w_k|}{|\wedge_{i=1}^n w_i|},
\end{equation}
and this is easily verified (or see for example Lemma 5.a in \cite{Sch}).

\end{proof}
%\begin{proof}
%Take into account that $(S_1,\dots, S_m)$ form a uniform cover of $[n]$ with weights $(p_1,\dots, p_m)$ if and only if $(S_1^c,\dots, S_m^c)$ form a uniform cover of $[n]$ with weights $(p_1^\prime,\dots,p_m^\prime)$, where %$p_j^\prime=\frac{p_j}{p-1}$, being $p=\sum_{j=1}^m p_j$.
%Directly applying Theorem \ref{thm:BrascampLiebSpannedInSjc} to this uniform covering we obtain the result.
%\end{proof}
%\begin{rmk}
%In the same way, Theorem \ref{thm:BrascampLiebSpannedInSjc} implies Theorem \ref{thm:BrascampLiebSpannedInSj}.
%\end{rmk}

\section{Loomis-Whitney type inequalities}\label{sec:Loomis-WhitneyInequalities}

In this section we prove the affine-invariant versions of Loomis-Whitney inequalities, as well as their functional versions for log-concave functions and their restricted versions. We start proving Theorem \ref{thm:AffineInvariantBollobasThomasonIntro}, which is a direct consequence of the various Brascamp-Lieb inequalities of the previous section.

\begin{proof}[Proof of Theorem \ref{thm:AffineInvariantBollobasThomasonIntro}]
To prove the first one, let $L_1=\{x\in\R^n\,:\, P_{\widetilde{H}_j^\perp} x\in P_{\widetilde{H}_j^\perp} K\;{\rm for\; all}\; 1\leqslant j\leqslant m\}$. Apply Theorem \ref{thm:AffineFinner} to the functions $f_j(x)=\chi_{P_{\widetilde{H}_j^\perp}}(x)$ and take into account that
$K\subseteq L_1$ and that $\prod_{j=1}^m f_j^{p_j}(P_{\widetilde{H}_j^\perp}x)=\chi_{L_1}$. The rest of the inequalities are proved in the same way.
\end{proof}
We can now upgrade the above geometric inequalities of Theorem \ref{thm:AffineInvariantBollobasThomasonIntro} to functional ones. The inequalities that we obtain are in two ways (affine invariance and general projections) a generalisation of the Gagliardo-Nirenberg inequality. In the same way, this is a generalisation of  Lemma 3.1 in \cite{BN} and Theorem 1.1 in \cite{CGS}.
\begin{thm}\label{thm:Gagliardo-NirembergExtended}
Let $\{w_i\}_{i=1}^n$ be a basis of $\R^n$ and let $(S_1,\dots, S_m)$ form a uniform cover of $[n]$ with weights $(p_1,\dots, p_m)$. If $H_j={\rm span}\{w_k\,:\,k\in S_j\}$, $\widetilde{H}_j={\rm span}\{w_k\,:\,k\not\in S_j\}$ and $p=\sum_{j=1}^mp_j$, then, for every $f:\R^n\to\R$ with compact support, continuous on its support we have the following four inequalities:

\begin{align}
\label{eq:Proj1}\Vert f\Vert_{p}\leqslant BL_1^{\frac{1}{p}}\prod_{j=1}^m\Vert P_{\widetilde{H}_j^\perp}|f|\Vert_1^\frac{p_j}{p}\\
\label{eq:Proj2}\Vert f\Vert_{\frac{p}{p-1}}\leqslant BL_2^{\frac{1}{p}}\prod_{j=1}^m\Vert P_{H_j^\perp}|f|\Vert_1^\frac{p_j}{p}\\
\label{eq:Proj3}\Vert f\Vert_{p}\leqslant BL_2^{\frac{1}{p}}\prod_{j=1}^m\Vert P_{H_j}|f|\Vert_1^\frac{p_j}{p}\\
\label{eq:Proj4}\Vert f\Vert_{\frac{p}{p-1}}\leqslant BL_1^{\frac{1}{p}}\prod_{j=1}^m\Vert P_{\widetilde{H}_j}|f|\Vert_1^\frac{p_j}{p}
\end{align}
\end{thm}
\begin{proof}
For any $0<t\leq\Vert f\Vert_\infty$, let us call $L_t$ the compact set
$$
L_t=\{x\in\R^n\,:\,|f(x)|\geq t\}.
$$
Notice that, by Fubini's theorem, for any $p>1$
$$
\int_{\R^n}|f(x)|^pdx=p\int_0^{\Vert f\Vert_\infty} t^{p-1}|L_t|dt\leq\left(\int_0^{\Vert f\Vert_\infty}|L_t|^\frac{1}{p}\right)^{p},
$$
where the last inequality follows from the fact that the functions
\begin{itemize}
\item $\displaystyle{f_1(s)=p\int_0^s t^{p-1}|L_t|dt}$
\item $\displaystyle{f_2(s)=\left(\int_0^s|L_t|^\frac{1}{p}dt\right)^p}$
\end{itemize}
verify that $f_1(0)=f_2(0)=0$ and $f_1^\prime(s)\leqslant f_2^\prime(s)$ for every $0<s<\Vert f\Vert_\infty$ and then for every $0\leqslant s\leqslant\Vert f\Vert_\infty$ we have $f_1(s)\leqslant f_2(s)$. In the same way, considering the functions
\begin{itemize}
\item $\displaystyle{g_1(s)=\frac{p}{p-1}\int_0^s t^\frac{1}{p-1}|L_t|dt}$
\item $\displaystyle{g_2(s)=\left(\int_0^s|L_t|^\frac{p-1}{p}dt\right)^\frac{p}{p-1}}$
\end{itemize}
we have that
$$
\int_{\R^n}|f(x)|^\frac{p}{p-1}dx=\frac{p}{p-1}\int_0^{\Vert f\Vert_\infty} t^\frac{1}{p-1}|L_t|dt\leq\left(\int_0^{\Vert f\Vert_\infty}|L_t|^\frac{p-1}{p}\right)^\frac{p}{p-1}.
$$
Therefore
\begin{itemize}
\item$\displaystyle{\Vert f\Vert_{p}\leqslant\int_0^{\Vert f\Vert_\infty}|L_t|^\frac{1}{p},}$
\item$\displaystyle{\Vert f\Vert_{\frac{p}{p-1}}\leqslant\int_0^{\Vert f\Vert_\infty}|L_t|^\frac{p-1}{p}.}$
\end{itemize}

By inequality \eqref{eq1AffineInvariantBT} in Theorem \ref{thm:AffineInvariantBollobasThomasonIntro} we have that
\begin{eqnarray*}
\int_0^{\Vert f\Vert_\infty}|L_t|^\frac{1}{p}dt&\leqslant&BL_1^\frac{1}{p}\int_0^{\Vert f\Vert_\infty}\prod_{j=1}^m|P_{\widetilde{H_j}^\perp}L_t|^\frac{p_j}{p}dt\cr
&\leqslant& BL_1^\frac{1}{p}\prod_{j=1}^m\left(\int_0^{\Vert f\Vert_\infty}|P_{\widetilde{H_j}^\perp}L_t|dt\right)^\frac{p_j}{p}.\cr
\end{eqnarray*}
Since for every $0<t\leq\Vert f\Vert_\infty=\Vert P_{\widetilde{H}_j^\perp}f\Vert_\infty$
$$
P_{\widetilde{H}_j^\perp}L_t=\{x\in \widetilde{H}_j^\perp\,:\,\sup_{y\in \widetilde{H}_j}|f(x+y)|\geqslant t\}=\{x\in \widetilde{H}_j^\perp\,:\,P_{\widetilde{H}_j^\perp}|f|(x)\geqslant t\},
$$
as a consequence of Fubini's theorem we obtain the first inequalities.

The other three inequalities are proved in the same way by applying the remaining inequalities in Theorem \ref{thm:AffineInvariantBollobasThomasonIntro}.
\end{proof}

\begin{rmk}
Notice that if $f\in\mathcal{C}^{(1)}(\R^n)$ with compact support then we have (see \cite{Z}, where a convexification of the sets $L_t$ is used) that for every $w\in S^{n-1}$
$$
\Vert\nabla_wf\Vert_1=2\int_0^\infty|P_{w^\perp}L_t|dt=2\Vert P_{w^\perp}|f|\Vert_1,
$$
and taking $S_j=\{j\}$ and $p_j=1$ for $1\leqslant j\leqslant n$ in the second inequality we obtain the generalisation of the Gagliardo-Nirenberg inequality proved in \cite[Theorem 5.1]{Z} corresponding to the case of $n$ vectors.
\end{rmk}

\begin{rmk}
Notice also that if we take $f(x)=\chi_K$ for some compact set we recover the inequalities in Theorem \ref{thm:AffineInvariantBollobasThomasonIntro}.
\end{rmk}

In the setting of log-concave functions we also have the following functional versions of the previous inequalities.

\begin{thm}\label{thm:AffineInvariantFunctionalBollobasThomason}
Let $\{w_i\}_{i=1}^n$ be a basis of $\R^n$ and let $(S_1,\dots, S_m)$ form a uniform cover of $[n]$ with weights $(p_1,\dots, p_m)$. If $H_j={\rm span}\{w_k\,:\,k\in S_j\}$, $\widetilde{H}_j={\rm span}\{w_k\,:\,k\not\in S_j\}$, $d_j={\rm dim}H_j=|S_j|$, $\widetilde{d}_j=n-d_j={\rm dim}\widetilde{H}_j$, and $p=\sum_{j=1}^mp_j$ then for every $f\in\mathcal{F}(\R^n)$ we have the following four inequalities:
\begin{align}
\label{eq1}\Vert f\Vert_\infty^{p-1}\Vert f\Vert_1&\leqslant BL_1\cdot\frac{n!}{\prod_{j=1}^m(d_j!)^{p_j}}\prod_{j=1}^m\Vert P_{\widetilde{H}_j^\perp} f\Vert_1^{p_j},\\
\label{eq2}\Vert f\Vert_\infty\Vert f\Vert_1^{p-1}&\leqslant BL_2\cdot \frac{(n!)^{p-1}}{\prod_{j=1}^m(\widetilde{d}_j!)^{p_j}}\prod_{j=1}^m\Vert P_{H_j^\perp} f\Vert_1^{p_j},\\
\label{eq3}\Vert f\Vert_\infty^{p-1}\Vert f\Vert_1&\leqslant BL_2\cdot \frac{n!}{\prod_{j=1}^m(d_j!)^{p_j}}\prod_{j=1}^m\Vert P_{H_j}f\Vert_1^{p_j},\\
\label{eq4}\Vert f\Vert_\infty\Vert f\Vert_1^{p-1}&\leqslant BL_1\cdot \frac{(n!)^{p-1}}{\prod_{j=1}^{m}(\widetilde{d}_j!)^{p_j}}\prod_{j=1}^m\Vert P_{\widetilde{H}_j}f\Vert_1^{p_j}.
\end{align}
\end{thm}
\begin{proof}
Let  $C$ be the set
$$
C=\{(x,t)\in\R^{n}\times[0,\infty): f(x)\geqslant e^{-t}\Vert f\Vert_\infty\}.
$$
Since $f$ is log-concave, we see that $C$ is convex. Besides,
\begin{align*}
\int_C e^{-t}\,dx\,dt&=\int_0^\infty e^{-t}\,\big|\{x\in\R^n:f(x)\geqslant e^{-t}\Vert f\Vert_\infty\}\big|\,dt\\
&=\int_0^1\big|\{x\in\R^n:f(x)\geqslant s\Vert f\Vert_\infty\}\big|\,ds\\
&=\int_{\R^n}\frac{f(x)}{\Vert f\Vert_\infty}dx.
\end{align*}
For any linear subspace $F$ of $\R^n$ let us denote $\overline{F}={\rm span}\{F, e_{n+1}\}$ and notice that %$\overline{E\cap H}=\overline{E}\cap\overline{H}$. Notice also that
\begin{eqnarray*}
\int_{P_{\overline{F}}C} e^{-t}\,dx\,dt&=&\int_0^\infty e^{-t}|\{x\in F:(x,t)\in P_{\overline{F}}C\}|dt\cr
&=&\int_0^\infty e^{-t}|\{x\in F:\sup_{y\in F^\perp}f(x+y)\geqslant e^{-t}\Vert f\Vert_\infty\}|dt\cr
&=&\int_0^1 |\{x\in F:P_F f(x)\geqslant s\Vert f\Vert_\infty\}|ds\cr
&=&\int_{F}\frac{P_Ff(x)}{\Vert f\Vert_\infty}dx.\cr
\end{eqnarray*}
By inequality \eqref{eq1AffineInvariantBT} in Theorem \ref{thm:AffineInvariantBollobasThomasonIntro},
\begin{eqnarray*}
&&\int_C e^{-t}\,dx\,dt=\int_0^\infty e^{-t}|\{x\in\R^n\,:\,f(x)\geqslant e^{-t}\Vert f\Vert_\infty\}|dt\cr
&\leqslant&BL_1\cdot\int_0^\infty e^{-t}\prod_{j=1}^m|P_{\widetilde{H}_j^\perp}\{x\in\R^n\,:\,f(x)\geqslant e^{-t}\Vert f\Vert_\infty\}|^{p_j}dt\cr
&=&BL_1\cdot\int_0^\infty e^{-t}\prod_{j=1}^m|\{x\in \widetilde{H}_j^\perp\,:\,P_{\widetilde{H}_j^\perp}f(x)\geqslant e^{-t}\Vert f\Vert_\infty\}|^{p_j}dt.\cr
\end{eqnarray*}
Since $\displaystyle{\sum_{j=1}^m\frac{p_jd_j}{n}=1}$, by H\"older's inequality, the latter integral is bounded above by
$$
\prod_{j=1}^m\left(\int_0^\infty e^{-t}|\{x\in \widetilde{H}_j^\perp\,:\,P_{\widetilde{H}_j^\perp}f(x)\geqslant e^{-t}\Vert f\Vert_\infty\}|^\frac{n}{d_j}dt\right)^\frac{p_jd_j}{n}.
$$
Note that the sets
$$
L_j=\big\{(x,t)\in \widetilde{H}_j^\perp\times[0,\infty)\,:\,P_{\widetilde{H}_j^\perp}f(x)\geqslant e^{-t}\Vert f\Vert_\infty\big\}
$$
are convex for all $1\leqslant j\leqslant m$. Therefore, by the Brunn-Minkowski inequality,
the functions
$$
h_j(t)=\big|\{x\in \widetilde{H}_j^\perp\,:\,P_{\widetilde{H}_j^\perp}f(x)\geqslant e^{-t}\Vert f\Vert_\infty\}\big|^\frac{1}{d_j}
$$
are concave and we can apply Lemma \ref{lem:0-dimBerwaldExponentialDensity} to get
\begin{equation*}
\left(\frac{1}{n!}\int_0^\infty e^{-t}h_j(t)^n\,dt\right)^\frac{p_jd_j}{n}
\leqslant \left(\frac{1}{d_j!}\int_0^\infty e^{-t}h_j(t)^{d_j}\,dt\right)^{p_j}.
\end{equation*}
Combine the above to get
\begin{eqnarray*}
\int_C e^{-t}\,dx\,dt
&\leqslant&BL_1\cdot\frac{n!}{\prod_{j=1}^m(d_j!)^{p_j}}\prod_{j=1}^m\left(\int_0^\infty e^{-t}h_j(t)^{d_j}\,dt\right)^{p_j}\cr
&=&BL_1\cdot\frac{n!}{\prod_{j=1}^m(d_j!)^{p_j}}\prod_{j=1}^m\left(\int_{\widetilde{H}_j^\perp}\frac{P_{\widetilde{H}_j^\perp}f(x)}{\Vert f\Vert_\infty}dx\right)^{p_j},
\end{eqnarray*}
which proves \eqref{eq1}. Then \eqref{eq2} is obtained from \eqref{eq1} by taking the uniform cover of $[n]$ consisting of the sets $(S_1^c,\dots, S_m^c)$ with weights $(p_1^\prime,\dots,p_m^\prime)$. In order to prove \eqref{eq3}, we apply the third instead of the first inequality in Theorem \ref{thm:AffineInvariantBollobasThomasonIntro} to obtain
$$
\int_C e^{-t}dxdt\leqslant BL_2\cdot\int_0^\infty e^{-t}\prod_{j=1}^m|\{x\in H_j\,:\,P_{H_j}f(x)\geqslant e^{-t}\Vert f\Vert_\infty\}|^{p_j}dt
$$
and argue in the same way. Finally, \eqref{eq4} is obtained from \eqref{eq3} again by taking the uniform cover of $[n]$ by the sets $(S_1^c,\dots, S_m^c)$ with weights $(p_1^\prime,\dots,p_m^\prime)$.
\end{proof}

\noindent\textit{Remark.}
Notice that if $K$ is a convex body containing the origin in its interior, while the Brascamp-Lieb inequality provides the inequalities in Theorem \ref{thm:AffineInvariantBollobasThomasonIntro} by taking the functions $f_j=\chi_{P_{H_j}K}$, this theorem provides them by taking $f(x)=e^{-\Vert x\Vert_K}$.

Moreover, using the inequality \eqref{eq1} for a log-concave probability density $p$ on $\mathbb{R}^n$ such that $p(0)=1$ and $p(\pm x_1,\ldots, \pm x_n)$ does not depend on the choice of signs, we prove the sharp version of \cite[Lemma 3.1]{BN}. More precisely, for $d_j=n-1$ and $p_j=\frac{1}{n-1}$, we obtain
$$\prod_{j=1}^n\int_{\{x_j=0\}}p(x)\, dx\geqslant\frac{n!}{n^n}.$$
This is sharp because we have equality for the density
$$p(x)=\exp\left(-2n!^{1/n}\max_{j\leqslant n}|x_j| \right).$$
In particular, we have the following, which extends \cite[Lemma 2.10]{AAGJV}.

\begin{cor}\label{cor:AffineFunctionalMin}
Let $\{w_i\}_{i=1}^n$ be a basis of $\R^n$ and let $(S_1,\dots, S_m)$ form a uniform cover of $[n]$ with weights $(p_1,\dots p_m)$. If $H_j={\rm span}\{w_k\,:\,k\in S_j\}$, $\widetilde{H}_j={\rm span}\{w_k\,:\,k\not\in S_j\}$, $d_j={\rm dim}H_j=|S_j|$,  $\widetilde{d}_j=n-d_j={\rm dim}\widetilde{H}_j$, and $p=\sum_{j=1}^mp_j$ then for any integrable log-concave functions $f_{1,j}: H_j\to[0,\infty)$, $\widetilde{f}_{1,j}:\widetilde{H}_{j}:\to[0,\infty)$, $f_{2,j}: H_j^\perp\to[0,\infty)$, $\widetilde{f}_{2,j}:\widetilde{H}_{j}^\perp:\to[0,\infty)$, we have the following four inequalities:

\begin{align}
\label{eq:4.1}\int_{\R^n}\min_{1\leqslant j\leqslant m}\left\{\frac{\widetilde{f}_{2,j}(P_{\widetilde{H}_j^\perp}x)}{\Vert\widetilde{f}_{2,j}\Vert_\infty}\right\}dx&\leqslant BL_1\cdot \frac{n!}{\prod_{j=1}^m(d_j!)^{p_j}}\prod_{j=1}^m\left(\int_{\widetilde{H}_j^\perp}\frac{\widetilde{f}_{2,j}(x)}{\Vert\widetilde{f}_{2,j}\Vert_\infty}dx\right)^{p_j},\\
\label{eq:4.2}\left(\int_{\R^n}\min_{1\leqslant j\leqslant m}\left\{\frac{f_{2,j}(P_{H_j^\perp}x)}{\Vert f_{2,j}\Vert_\infty}\right\}dx\right)^{p-1}&\leqslant BL_2\cdot \frac{(n!)^{p-1}}{\prod_{j=1}^m(\widetilde{d}_j!)^{p_j}}\prod_{j=1}^m\left(\int_{H_j^\perp}\frac{f_{2,j}(x)}{\Vert f_{2,j}\Vert_\infty}dx\right)^{p_j},\\
\label{eq:4.3}\int_{\R^n}\min_{1\leqslant j\leqslant m}\left\{\frac{f_{1,j}(P_{H_j}x)}{\Vert f_{1,j}\Vert_\infty}\right\}dx&\leqslant BL_2\cdot \frac{n!}{\prod_{j=1}^m(d_j!)^{p_j}}\prod_{j=1}^m\left(\int_{H_j}\frac{f_{1,j}(x)}{\Vert f_{1,j}\Vert_\infty}dx\right)^{p_j},\\
\label{eq:4.4}\left(\int_{\R^n}\min_{1\leqslant j\leqslant m}\left\{\frac{\widetilde{f}_{1,j}(P_{\widetilde{H}_j}x)}{\Vert\widetilde{f}_{1,j}\Vert_\infty}dx\right\}\right)^{p-1}&\leqslant BL_1\cdot \frac{(n!)^{p-1}}{\prod_{j=1}^{m}(\widetilde{d}_j!)^{p_j}}\prod_{j=1}^m\left(\int_{\widetilde{H}_j}\frac{\widetilde{f}_{1,j}(x)}{\Vert\widetilde{f}_{1,j}\Vert_\infty}dx \right)^{p_j}.
\end{align}
\end{cor}

\begin{proof} Simply take into account that $f(x)=\min_{1\leqslant j\leqslant m}\left\{\frac{\widetilde{f}_{2,j}(P_{\widetilde{H}_j^\perp}x)}{\Vert\widetilde{f}_{2,j}\Vert_\infty}\right\}$ verifies that
$P_{\widetilde{H}_j^\perp}f\leqslant\frac{\widetilde{f}_{2,j}}{\Vert \widetilde{f}_{2,j}\Vert_\infty}$.\end{proof}

\smallskip

Next we prove restricted versions of these inequalities. We will first prove a functional version for log-concave functions and obtain the geometric version as a consequence of it.

\begin{thm}\label{thm:AffineLoomisWhitneyFunctionalLocal}
Let $\{w_i\}_{i=1}^n$ be a basis of $\R^n$ and let $S\subseteq[n]$ with cardinality $|S|=d$. Let $(S_1,\dots, S_m)$ form a uniform cover of $S$ with weights $(p_1,\dots, p_m)$. If $H={\rm span}\{w_k\,:\,k\in S\}$, $H_j={\rm span}\{w_k\,:\,k\in S_j\}$, $\widetilde{H}_j={\rm span}\{w_k\,:\,k\in S\setminus S_j\}$, $d_j={\rm dim}H_j=|S_j|$, $\widetilde{d}_j=d-d_j={\rm dim}\widetilde{H}_j$, and $p=\sum_{j=1}^mp_j$ then for every $f\in\mathcal{F}(\R^n)$ we have the following four inequalities:
\begin{align}
\label{eq:4.9}\Vert P_{H^\perp}f\Vert_1^{p-1}\Vert f\Vert_1&\leqslant\frac{d!\prod_{j=1}^m|\wedge_{k\in S\setminus S_j}w_k|^{p_j}}{|\wedge_{i\in S}w_i|^{p-1}\prod_{j=1}^m(d_j!)^{p_j}}\prod_{j=1}^m\Vert P_{\widetilde{H}_j^\perp}f\Vert_1^{p_j},\\
\label{eq:4.10}\Vert P_{H^\perp}f\Vert_1\Vert f\Vert_1^{p-1}&\leqslant\frac{(d!)^{p-1}\prod_{j=1}^m|\wedge_{k\in S_j}w_k|^{p_j}}{|\wedge_{i\in S}w_i|\prod_{j=1}^m(\widetilde{d}_j!)^{p_j}}\prod_{j=1}^m\Vert P_{H_j^\perp}f\Vert_1^{p_j},\\
\label{eq:4.11}\Vert P_{H^\perp}f\Vert_1^{p-1}\Vert f\Vert_1&\leqslant\frac{d!\prod_{j=1}^m|\wedge_{k\in S_j}w_k|^{p_j}}{|\wedge_{i\in S}w_i|\prod_{j=1}^m(d_j!)^{p_j}}\prod_{j=1}^m\Vert P_{H_j\oplus H^\perp}f\Vert_1^{p_j},\\
\label{eq:4.12}\Vert P_{H^\perp}f\Vert_1\Vert f\Vert_1^{p-1}&\leqslant\frac{(d!)^{p-1}\prod_{j=1}^m|\wedge_{k\in S\setminus S_j}w_k|^{p_j}}{|\wedge_{i\in S}w_i|^{p-1}\prod_{j=1}^m(\widetilde{d}_j!)^{p_j}}\prod_{j=1}^m\Vert P_{\widetilde{H}_j\oplus H^\perp}f\Vert_1^{p_j}.
\end{align}
\end{thm}

\begin{proof}
Let $C$ be the set
$$
C=\{(x,t)\in\R^{n}\times[0,\infty): f(x)\geqslant e^{-t}\Vert f\Vert_\infty\}
$$
and for any linear subspace $F$ denote $\overline{F}={\rm span}\{F, e_{n+1}\}$. We have that
$$
\int_C e^{-t}dxdt=\int_{\R^n}\frac{f(x)}{\Vert f\Vert_\infty}dx \quad {\rm and}\quad\int_{P_{\overline{F}}C} e^{-t}dxdt=\int_{F}\frac{P_Ff(x)}{\Vert f\Vert_\infty}dx.
$$

Notice that, by inequality \eqref{eq1AffineInvariantBT} in Theorem \ref{thm:AffineInvariantBollobasThomasonIntro},
\begin{eqnarray*}
&&\hspace*{-0.5cm}\int_C e^{-t}dxdt=\int_{P_{\overline{H^\perp}}C}e^{-t}|C\cap((x,t)+H)|dtdx\cr
&\leqslant&\frac{\prod_{j=1}^m|\wedge_{k\in S\setminus S_j}w_k|^{p_j}}{|\wedge_{i\in S}w_i|^{p-1}}\int_{P_{\overline{H^\perp}}C}e^{-t}\prod_{j=1}^m|P_{(x,t)+(\widetilde{H}_j^\perp\cap H)}(C\cap((x,t)+H))|^{p_j}dtdx\cr
&=&\frac{\prod_{j=1}^m|\wedge_{k\in S\setminus S_j}w_k|^{p_j}}{|\wedge_{i\in S}w_i|^{p-1}}\int_{P_{\overline{H^\perp}}C}e^{-t}\prod_{j=1}^m|P_{\widetilde{H}_j^\perp}C\cap((x,t)+(\widetilde{H}_j^\perp\cap H))|^{p_j}dtdx.
\end{eqnarray*}

Since $\displaystyle{\sum_{j=1}^m\frac{p_jd_j}{d}=1}$, by H\"older's inequality, we have that
\begin{eqnarray*}
&&\hspace*{-0.5cm}\frac{1}{\int_{H^\perp}\frac{P_{H^\perp}f(x)}{\Vert f\Vert_\infty}dx}\int_{P_{\overline{H^\perp}}C}e^{-t}\prod_{j=1}^m|P_{\widetilde{H}_j^\perp}C\cap((x,t)+(\widetilde{H}_j^\perp\cap H))|^{p_j}dtdx\cr
&\leqslant&\prod_{j=1}^m\left(\frac{1}{\int_{H^\perp}\frac{P_{H^\perp}f(x)}{\Vert f\Vert_\infty}dx}\int_{P_{\overline{H^\perp}}C}e^{-t}|P_{\overline{\widetilde{H}_j^\perp}}C\cap((x,t)+(\widetilde{H}_j^\perp\cap H))|^\frac{d}{d_j}dtdx\right)^\frac{p_jd_j}{d}.
\end{eqnarray*}
The sets
$$
L_j=P_{\overline{\widetilde{H}_j^\perp}}C=\{(x,t)\in \widetilde{H}_j^\perp\times[0,\infty)\,:\,P_{\widetilde{H}_j^\perp}f(x)\geqslant e^{-t}\Vert f\Vert_\infty\}
$$
are convex, therefore, by the Brunn-Minkowski inequality, the functions
$$
h_j(x,t)=|P_{\overline{\widetilde{H}_j^\perp}}C\cap((x,t)+(\widetilde{H}_j^\perp\cap H))|^\frac{1}{d_j}
$$
are concave. We can now use Lemma \ref{lem:BerwaldExponentialDensity} and get that the $j$-th term of the above product is at most
$$
\left(\frac{1}{d_j!\int_{H^\perp}\frac{P_{H^\perp}f(x)}{\Vert f\Vert_\infty}dx}\int_{P_{\overline{H^\perp}}C}e^{-t}h_j(x,t)^{d_j}\,dt\,dx\right)^{p_j},$$
for every $1\leqslant j\leqslant m$. However, the last integral is equal to
$$
\left(\frac{1}{d_j!\int_{H^\perp}\frac{P_{H^\perp}f(x)}{\Vert f\Vert_\infty}dx}\int_{P_{\overline{\widetilde{H}_j^\perp}}C}e^{-t}dt\,dx\right)^{p_j}
=\left(\frac{1}{d_j!\int_{H^\perp}\frac{P_{H^\perp}f(x)}{\Vert f\Vert_\infty}dx}\int_{\widetilde{H}_j^\perp}\frac{P_{\widetilde{H}_j^\perp}f(x)}{\Vert f\Vert_\infty} dx\right)^{p_j}.$$
Thus,
$$
\int_C e^{-t}dt\leqslant\frac{d!\prod_{j=1}^m|\wedge_{k\in S\setminus S_j}w_k|^{p_j}}{|\wedge_{i\in S}w_i|^{p-1}\prod_{j=1}^m(d_j!)^{p_j}}\frac{\prod_{j=1}^m\left(\int_{\widetilde{H}_j^\perp}\frac{P_{\widetilde{H}_j^\perp}f(x)}{\Vert f\Vert_\infty} dx\right)^{p_j}}{\left(\int_{H^\perp}\frac{P_{H^\perp}f(x)}{\Vert f\Vert_\infty}dx\right)^{p-1}},
$$
which proves \eqref{eq:4.9}. As before, \eqref{eq:4.10} is obtained from \eqref{eq:4.9} by taking into account that $(S\setminus S_1,\dots S\setminus S_m)$ is a uniform cover of $S$ with weights $(p_1^\prime,\dots,p_m^\prime)$. In order to prove \eqref{eq:4.11} we apply inequality \eqref{eq3AffineInvariantBT} of Theorem \ref{thm:AffineInvariantBollobasThomasonIntro} to obtain
$$
\int_C e^{-t}dxdt\leqslant\frac{\prod_{j=1}^m|\wedge_{k\in S_j}w_k|^{p_j}}{|\wedge_{i\in S}w_i|}\int_{P_{\overline{H^\perp}}C}e^{-t}\prod_{j=1}^m|P_{\overline{H_j\oplus H^\perp}}C\cap((x,t)+H_j)|^{p_j}dtdx
$$
and argue in the same way as for \eqref{eq:4.9}.
\end{proof}

If  $K$ a convex body containing the origin, applying the inequalities of Theorem \ref{thm:AffineLoomisWhitneyFunctionalLocal} to the function $e^{-\Vert x\Vert_K}$, we prove Theorem \ref{thm:AffineGeometricLocalLoomisWhitneyIntro}. Let us point out that Theorem \ref{thm:AffineGeometricLocalLoomisWhitneyIntro} can be proved directly, without making use of Theorem \ref{thm:AffineLoomisWhitneyFunctionalLocal}, by using Theorem \ref{berwald01} and the same technique as in the proof of Theorem \ref{thm:AffineLoomisWhitneyFunctionalLocal}.

\section{Dual Loomis-Whitney type inequalities}\label{sec:DualLoomis-WhiteneyInequalities}

In this section we will prove dual Loomis-Whitney inequalities, as well as functional and local versions of them. We start with the following consequence of the reverse Brascamp-Lieb inequalities for log-concave functions, which extends Theorem 1.4 in \cite{L}. We note also that $f(0)$ does not play any role in the following theorem, while it did in \cite{L}.

\begin{thm}\label{thm:ReverseBL1Function}
Let $w_1,\dots,w_n$ be $n$ vectors spanning $\R^n$, let $m\geqslant 1$ and let $(S_1,\dots, S_m)$ be a uniform cover of $[n]$ with weights $(p_1,\dots, p_m)$. Let $H_j={\rm span}\{w_k\,:\,k\in S_j\}$, $\widetilde{H}_j={\rm span}\{w_k\,:\,k\not\in S_j\}$, $d_j={\rm dim}H_j$, $\widetilde{d}_j={\rm dim}\widetilde{H}_j=n-d_j$, and $p=\sum_{j=1}^mp_j$. Then, for every $f\in\mathcal{F}(\R^n)$ we have the following inequalities:
\begin{align}
\label{eq:5.1}\int_{\R^n}f^n(x)dx&\geqslant\frac{1}{BL_1}\frac{\prod_{j=1}^m(d_j)^{p_jd_j}}{n^n}\prod_{j=1}^m\left(\int_{\widetilde{H}_j^\perp}f^{d_j}(x)dx\right)^{p_j},\\
\label{eq:5.2}\left(\int_{\R^n}f^n(x)dx\right)^{p-1}&\geqslant\frac{1}{BL_2}\frac{\prod_{j=1}^m(\widetilde{d}_j)^{p_j\widetilde{d}_j}}{n^{n(p-1)}}\prod_{j=1}^m\left(\int_{H_j^\perp}f^{\widetilde{d}_j}(x)dx\right)^{p_j}\\
\label{eq:5.3}\int_{\R^n}f^n(x)dx&\geqslant\frac{1}{BL_2}\frac{\prod_{j=1}^m(d_j)^{p_jd_j}}{n^n}\prod_{j=1}^m\left(\int_{H_j}f^{d_j}(x)dx\right)^{p_j}\\
\label{eq:5.4}\left(\int_{\R^n}f^n(x)dx\right)^{p-1}&\geqslant\frac{1}{BL_1}\frac{\prod_{j=1}^m(\widetilde{d}_j)^{p_j\widetilde{d}_j}}{n^{n(p-1)}}\prod_{j=1}^m\left(\int_{\widetilde{H}_j}f^{\widetilde{d}_j}(x)dx\right)^{p_j}.
\end{align}
\end{thm}

\begin{proof}
For the first one, let $g(x)=f^n(x/n)$ and for every $1\leqslant j\leqslant m$, let $g_j(x)=f^{d_j}(x/d_j)$ for all $x\in\widetilde{H}_j^\perp$. For every $x_j\in \widetilde{H}_j^\perp$, if we write $x=\sum_{j=1}^mp_jx_j$, we have that $x=\sum_{j=1}^m p_jd_j y_j$ with $y_j=x_j/d_j$. Then, notice that from $I_n=\sum_{j=1}^mp_jP_{H_j^\prime}$, we have that $\sum_{j=1}^m p_jd_j=n$. Since $f$ is log-concave,
$$
g(x)=f^n\left(\frac{\sum_{j=1}^mp_jd_jy_j}{n}\right)\geqslant\prod_{j=1}^mf^{p_jd_j}(y_j)=\prod_{j=1}^mg^{p_j}(x_j).
$$
By Theorem \ref{thm:AffineFinner} we have that
$$
\int_{\R^n}f^n\left(\frac{y}{n}\right)dx\geqslant\frac{|\wedge_{i=1}^nw_i|^{p-1}}{\prod_{j=1}^m|\wedge_{k\not\in S_j}w_k|^{p_j}}\prod_{j=1}^m\left(\int_{H_j}f^{d_j}\left(\frac{y}{d_j}\right)dx\right)^{p_j}.
$$
Making the change of variables $y=nx$ in the integral on the left hand side and $y=d_jx$ in each of the integrals of the right hand side we obtain the result.

Taking into account that $(S_1^c,\dots, S_m^c)$ forms a uniform cover of $[n]$ with weights $(p_1^\prime,\dots,p_m^\prime)$, where $p_j^\prime=\frac{p_j}{p-1}$, and applying \eqref{eq:5.1} to the subspaces $\widetilde{H}_j$ we obtain \eqref{eq:5.2}. The last two inequalites,
\eqref{eq:5.3} and \eqref{eq:5.4}, are proved in the same way, by using the reverse Brascamp-Lieb in Theorem \ref{thm:BrascampLiebSpannedInSj}.
\end{proof}

\noindent\textit{Remark.} If $K$ is a convex body containing the origin, applying the latter theorem to the function $f(x)=e^{-\Vert x\Vert_K}$ we obtain Theorem \ref{thm:DualBollobasThomasonIntro}.
\bigskip

The following theorem is also a consequence of the reverse Brascamp-Lieb inequality. It provides inequalities in the spirit of Theorem \ref{thm:ReverseBL1Function} with no powers of the functions involved.

\begin{thm}\label{thm:ReverseBL1FunctionWithoutPowers}
Let $w_1,\dots,w_n$ be $n$ vectors spanning $\R^n$, let $m\geqslant 1$ and let $(S_1,\dots, S_m)$ be a uniform cover of $[n]$ with weights $(p_1,\dots, p_m)$. Let $H_j={\rm span}\{w_k\,:\,k\in S_j\}$, $\widetilde{H}_j={\rm span}\{w_k\,:\,k\not\in S_j\}$, $d_j={\rm dim}H_j=|S_j|$, $\widetilde{d}_j={\rm dim}\widetilde{H}_j=n-d_j$, and $p=\sum_{j=1}^mp_j$. Then, for every $f\in\mathcal{F}(\R^n)$ we have the following inequalities:
\begin{align}
\label{eq:5.5}\Vert f\Vert_\infty^{p-1}\Vert f\Vert_1&\geqslant\frac{1}{BL_1}\cdot\frac{\prod_{j=1}^m(d_j)^{p_jd_j}}{n^n}\prod_{j=1}^m\Vert f\restr{\widetilde{H}_{j}^\perp}\Vert_1^{p_j}\\
\label{eq:5.6}\Vert f\Vert_\infty\Vert f\Vert_1^{p-1}&\geqslant\frac{1}{BL_2}\cdot\frac{\prod_{j=1}^m(\widetilde{d}_j)^{p_j\widetilde{d}_j}}{n^{n(p-1)}}\prod_{j=1}^m\Vert f\restr{H_j^\perp}\Vert_1^{p_j}\\
\label{eq:5.7}\Vert f\Vert_\infty^{p-1}\Vert f\Vert_1&\geqslant\frac{1}{BL_2}\cdot\frac{\prod_{j=1}^m(d_j)^{p_jd_j}}{n^n}\prod_{j=1}^m\Vert f\restr{H_j}\Vert_1^{p_j}\\
\label{eq:5.8}\Vert f\Vert_\infty\Vert f\Vert_1^{p-1}&\geqslant\frac{1}{BL_1}\cdot\frac{\prod_{j=1}^m(\widetilde{d}_j)^{p_j\widetilde{d}_j}}{n^{n(p-1)}}\prod_{j=1}^m\Vert f\restr{\widetilde{H}_j}\Vert_1^{p_j}.
\end{align}
\end{thm}
\begin{proof}
Let $g(x)=\frac{f\left(\frac{x}{n}\right)}{\Vert f\Vert_\infty}$ and for every $1\leqslant j\leqslant m$, let $g_j(x)=\frac{f\left(\frac{x}{d_j}\right)}{\Vert f\Vert_\infty}$ for all $x\in\widetilde{H}_j^{\perp}$. If we write $x=\sum_{j=1}^mp_jx_j$ with $x_j\in \widetilde{H}_j^\perp$, then we have that $x=\sum_{j=1}^m p_jd_j y_j$ with $y_j=x_j/d_j$. Notice that from $I_n=\sum_{j=1}^mp_jP_{H_j^\prime}$, we get $\sum_{j=1}^m p_jd_j=n$. Since $f$ is log-concave we obtain that
\begin{eqnarray*}
g(x)&=&\frac{f\left(\frac{\sum_{j=1}^mp_jd_jy_j}{n}\right)}{\Vert f\Vert_\infty}\geqslant\prod_{j=1}^m\left(\frac{f(y_j)}{\Vert f\Vert_\infty}\right)^\frac{p_jd_j}{n}\geqslant \prod_{j=1}^m\left(\frac{f(y_j)}{\Vert f\Vert_\infty}\right)^{p_j}=\prod_{j=1}^m g_j^{p_j}(x_j).
\end{eqnarray*}
Therefore, by Theorem \ref{thm:AffineFinner} we have that
$$
\Vert f\Vert_\infty^{p-1}\int_{\R^n}f\left(\frac{y}{n}\right)dy\geqslant\frac{|\wedge_{i=1}^nw_i|^{p-1}}{\prod_{j=1}^m|\wedge_{k\not\in S_j}w_k|^{p_j}}\prod_{j=1}^m\left(\int_{\widetilde{H}_j^\perp}f\left(\frac{y}{d_j}\right)dy\right)^{p_j}.
$$
Making the change of variables $y=nx$ in the integral on the left hand side and $y=d_jx$ in each of the integrals of the right hand side we obtain \eqref{eq:5.5}. Then \eqref{eq:5.6} follows from \eqref{eq:5.5} by taking into account that $(S_1^c,\dots, S_m^c)$ forms a uniform cover of $[n]$ with weights $(p_1^\prime,\dots,p_m^\prime)$. The last two inequalities, \eqref{eq:5.7} and \eqref{eq:5.8} are proved in the same way, by using Theorem \ref{thm:BrascampLiebSpannedInSj} instead of Theorem \ref{thm:AffineFinner}.
\end{proof}

Applying the latter inequalities to the function $f\colon H\to[0,\infty)$ given by $f(x)=|K\cap(x+H^\perp)|$, which is log-concave, we obtain Theorem \ref{thm:RestrictedDualBollobasThomason1Intro}.

If the function $f$ attains its maximum at the origin and all the weights $p_j$ are equal $p_j=\frac{p}{m}$, the following inequalities can be proved, with a better value of the constant.

\begin{thm}\label{thm:ReverseBL2FunctionWithoutPowers}
Let $w_1,\dots,w_n$ be $n$ vectors spanning $\R^n$, let $m\geqslant 1$ and let $(S_1,\dots, S_m)$ be a uniform cover of $[n]$ with equal weights $(\frac{p}{m},\dots, \frac{p}{m})$. Let $H_j={\rm span}\{w_k\,:\,k\in S_j\}$, $\widetilde{H}_j={\rm span}\{w_k\,:\,k\not\in S_j\}$, $d_j={\rm dim}H_j=|S_j|$, and $\widetilde{d}_j={\rm dim}\widetilde{H}_j=n-d_j$. Then, for every $f\in\mathcal{F}(\R^n)$ with $\Vert f\Vert_\infty=f(0)$ we have the following inequalities:
\begin{align}
\label{eq:5.9}\Vert f\Vert_\infty^{p-1}\Vert f\Vert_1&\geqslant\frac{|\wedge_{i=1}^nw_i|^{p-1}\prod_{j=1}^m(d_j!)^\frac{p}{m}}{\Gamma\left(1+\frac{nm}{p}\right)^\frac{p}{m}\prod_{j=1}^m|\wedge_{k\not\in S_j}w_k|^\frac{p}{m}}\prod_{j=1}^m\Vert f\restr{\widetilde{H}_j^\perp}\Vert_1^\frac{p}{m}\\
\label{eq:5.10}\Vert f\Vert_\infty\Vert f\Vert_1^{p-1}&\geqslant\frac{|\wedge_{i=1}^nw_i|\prod_{j=1}^m(\widetilde{d}_j!)^\frac{p}{m}}{\Gamma\left(1+\frac{nm(p-1)}{p}\right)^\frac{p}{m}\prod_{j=1}^m|\wedge_{k\in S_j}w_k|^\frac{p}{m}}\prod_{j=1}^m\Vert f\restr{H_j^\perp}\Vert_1^\frac{p}{m}\\
\label{eq:5.11}\Vert f\Vert_\infty^{p-1}\Vert f\Vert_1&\geqslant\frac{|\wedge_{i=1}^nw_i|\prod_{j=1}^m(d_j!)^\frac{p}{m}}{\Gamma\left(1+\frac{nm}{p}\right)^\frac{p}{m}\prod_{j=1}^m|\wedge_{k\in S_j}w_k|^\frac{p}{m}}\prod_{j=1}^m\Vert f\restr{H_j}\Vert_1^\frac{p}{m}\\
\label{eq:5.12}\Vert f\Vert_\infty\Vert f\Vert_1^{p-1}&\geqslant\frac{|\wedge_{i=1}^nw_i|^{p-1}\prod_{j=1}^m(\widetilde{d}_j!)^\frac{p}{m}}{\Gamma\left(1+\frac{nm(p-1)}{p}\right)^\frac{p}{m}\prod_{j=1}^m|\wedge_{k\not\in S_j}w_k|^\frac{p}{m}}\prod_{j=1}^m\Vert f\restr{\widetilde{H}_j}\Vert_1^\frac{p}{m}.
\end{align}
\end{thm}

\begin{proof}
Let $K_t$, $t>0$ be the convex body
$$
K_t=\{x\in\R^n\,:\,f(x)\geqslant e^{-t}\Vert f\Vert_\infty\}.
$$
Since $\Vert f\Vert_\infty=f(0)$, we have that $0\in K_t$ for every $t>0$.
By inequality \eqref{{eq1:DualBollobasIntro}} in Theorem \ref{thm:DualBollobasThomasonIntro}, we have that
\begin{align*}
\int_0^\infty e^{-t}|K_t|^\frac{m}{p}dt&\geqslant\frac{|\wedge_{i=1}^nw_i|^\frac{m(p-1)}{p}\prod_{j=1}^md_j!}{(n!)^\frac{m}{p}\prod_{j=1}^m|\wedge_{k\not\in S_j}w_k|}\int_0^\infty e^{-t}\prod_{j=1}^m|K_t\cap\widetilde{H}_j^\perp|dt\\
&=\frac{|\wedge_{i=1}^nw_i|^\frac{m(p-1)}{p}\prod_{j=1}^md_j!}{(n!)^\frac{m}{p}\prod_{j=1}^m|\wedge_{k\not\in S_j}w_k|}\int_{\widetilde{H}_1^\perp}\dots\int_{\widetilde{H}_{m}^\perp}\min_{1\leqslant j\leqslant m}\left\{\frac{f(x_j)}{\Vert f\Vert_\infty}\right\}dx_1\dots dx_m\\
&\geqslant\frac{|\wedge_{i=1}^nw_i|^\frac{m(p-1)}{p}\prod_{j=1}^md_j!}{(n!)^\frac{m}{p}\prod_{j=1}^m|\wedge_{k\not\in S_j}w_k|}\prod_{j=1}^m\int_{\widetilde{H}_j^\perp}\frac{f(x)}{\Vert f\Vert_\infty} dx.
\end{align*}
The set
$$
C=\{(x,t)\in\R^n\times[0,\infty)\,:\,f(x)\geqslant e^{-t}\Vert f\Vert_\infty\}
$$
is convex, so by Brunn-Minkowski inequality the function $h(t)=|K_t|^\frac{1}{n}$ is concave. Moreover, $\frac{p}{m}\leqslant1$, so using Lemma \ref{lem:0-dimBerwaldExponentialDensity}, we have that
\begin{eqnarray*}
\int_0^\infty e^{-t}|K_t|^\frac{m}{p}dt &\leqslant & \frac{\Gamma\left(1+\frac{nm}{p}\right)}{(n!)^\frac{m}{p}}\left(\int_0^\infty e^{-t}|K_t|\, dt\right)^\frac{m}{p}\\
&=&\frac{\Gamma\left(1+\frac{nm}{p}\right)}{(n!)^\frac{m}{p}}\left(\int_{\R^n}\frac{f(x)}{\Vert f\Vert_\infty}dx\right)^\frac{m}{p}.
\end{eqnarray*}
Therefore,
$$
\left(\int_{\R^n}\frac{f(x)}{\Vert f\Vert_\infty}dx\right)^\frac{m}{p}\geqslant\frac{|\wedge_{i=1}^nw_i|^\frac{m(p-1)}{p}\prod_{j=1}^md_j!}{\Gamma\left(1+\frac{nm}{p}\right)\prod_{j=1}^m|\wedge_{k\not\in S_j}w_k|}\prod_{j=1}^m\int_{\widetilde{H}_j^\perp}\frac{f(x)}{\Vert f\Vert_\infty} dx,
$$
which proves the first inequality. Then \eqref{eq:5.10} follows from \eqref{eq:5.9} by taking into account that $(S_1^c,\dots, S_m^c)$ forms a uniform cover of $[n]$ with equal weights $\frac{p}{m(p-1)}$. The last two inequalities, \eqref{eq:5.11} and \eqref{eq:5.12} are proved in the same way, by using inequality \eqref{eq3:DualBollobasIntro} instead of inequality \eqref{eq1:DualBollobasIntro} in
Theorem \ref{thm:DualBollobasThomasonIntro}.
\end{proof}

\begin{rmk}
Notice that if $K$ is a convex body containing the origin and $\frac{p}{m}=1$, applying \eqref{eq:5.9} and \eqref{eq:5.11} to the function $f(x)=e^{-\Vert x\Vert_K}$ we recover the first and the third inequality in Theorem \ref{thm:DualBollobasThomasonIntro} and if $\frac{p}{m(p-1)}=1$, we recover the second and the fourth ones in Theorem \ref{thm:DualBollobasThomasonIntro}.
\end{rmk}

Applying Theorem \ref{thm:ReverseBL2FunctionWithoutPowers} to the function $f\colon H\to[0,\infty)$ given by $f(x)=|K\cap(x+H^\perp)|$, which is log-concave, we obtain the following.

\begin{thm}\label{thm:RestrictedDualLoomisWhitney2}
Let $\{w_i\}_{i=1}^n$ be a basis of $\R^n$ and let $S\subseteq[n]$ with cardinality $|S|=d$. Let $(S_1,\dots, S_m)$ form a uniform cover of $S$ with equal weights $(\frac{p}{m},\dots, \frac{p}{m})$. If $H={\rm span}\{w_k\,:\,k\in S\}$, $H_j={\rm span}\{w_k\,:\,k\in S_j\}$, $\widetilde{H}_j={\rm span}\{w_k\,:\,k\in S\setminus S_j\}$, $d_j={\rm dim}H_j=|S_j|$, $\widetilde{d}_j=d-d_j={\rm dim}\widetilde{H}_j$, and $p=\sum_{j=1}^mp_j$ then for every convex body $K\subseteq\R^n$ such that
$\max_{x\in H}|K\cap(x+H^\perp)|=|K\cap H^\perp|$ we have the following inequalities:
\begin{align*}
|K\cap H^\perp|^{p-1}|K|&\geqslant\frac{|\wedge_{i\in S}w_i|^{p-1}\prod_{j=1}^m(d_j!)^\frac{p}{m}}{\Gamma\left(1+\frac{nm}{p}\right)^\frac{p}{m}\prod_{j=1}^m|\wedge_{k\in S\setminus S_j}w_k|^\frac{p}{m}}\prod_{j=1}^m|K\cap \widetilde{H}_j^\perp|^\frac{p}{m}\\
|K\cap H^\perp||K|^{p-1}&\geqslant\frac{|\wedge_{i\in S}w_i|\prod_{j=1}^m(\widetilde{d}_j!)^\frac{p}{m}}{\Gamma\left(1+\frac{nm(p-1)}{p}\right)^\frac{p}{m}\prod_{j=1}^m|\wedge_{k\in S_j}w_k|^\frac{p}{m}}\prod_{j=1}^m|K\cap H_j^\perp|^\frac{p}{m}\\
|K\cap H^\perp|^{p-1}|K|&\geqslant\frac{|\wedge_{i\in S}w_i|\prod_{j=1}^m(d_j!)^\frac{p}{m}}{\Gamma\left(1+\frac{nm}{p}\right)^\frac{p}{m}\prod_{j=1}^m|\wedge_{k\in S_j}w_k|^\frac{p}{m}}\prod_{j=1}^m|K\cap(H_j\oplus H^\perp)|^\frac{p}{m}\\
|K\cap H^\perp||K|^{p-1}&\geqslant\frac{|\wedge_{i\in S}w_i|^{p-1}\prod_{j=1}^m(\widetilde{d}_j!)^\frac{p}{m}}{\Gamma\left(1+\frac{nm(p-1)}{p}\right)^\frac{p}{m}\prod_{j=1}^m|\wedge_{k\in S\setminus S_j}w_k|^\frac{p}{m}}
    \prod_{j=1}^m|K\cap (\widetilde{H}_j\oplus H^\perp)|^\frac{p}{m}.
\end{align*}
\end{thm}
\begin{rmk}
If $S=\{1,2\}$, $S_1=\{1\}$, $S_2=\{2\}$, then $\frac{p}{m}=1$, $p=2$, $m=2$ and we obtain different
extensions of \eqref{eq:dualRestrictedAAGJV}.
\end{rmk}

\section{Appendix}
Here we present the proof of Theorem \ref{thm:AffineFinner2} which is totally different from the one we gave for Theorem \ref{thm:AffineFinner}. Let us note that this proof cannot be applied to prove directly Theorem \ref{thm:AffineFinner}; likewise the proof we gave of Theorem \ref{thm:AffineFinner} in Section \ref{sec:Brascamp-LiebInequalities} cannot be applied to prove directly Theorem \ref{thm:AffineFinner2}.
\begin{proof}
	Let $A:=\sum_{i=1}^nw_i\otimes w_i$, which is a symmetric positive definite matrix. Therefore it has a symmetric positive definite square root and
	$$
	I_n=\sum_{i=1}^n A^{-\frac{1}{2}}w_i\otimes A^{-\frac{1}{2}}w_i.
	$$
	Since the $n$ vectors $(w_i^\prime)_{i=1}^n=(A^{-\frac{1}{2}}w_i)_{i=1}^n$ provide a decomposition of the identity in $\R^n$ we have that they form an orthonormal basis in $\R^n$. Therefore,
	$$
	{\rm det}A^\frac{1}{2}=\wedge_{i=1}^nA^\frac{1}{2}w_i^\prime=\wedge_{i=1}^nw_i,
	$$
	and hence
	$$
	{\rm det}A^{-\frac{1}{2}}=\frac{1}{\wedge_{i=1}^nw_i}.
	$$
	
	For each $1\leqslant j\leqslant m$, let us denote  $H_j^\prime={\rm span}\{w_k'\,:\,k\in S_j\}$ and write $A_j^\frac{1}{2}$ for the restriction of $A^\frac{1}{2}$ as an operator from $H_j'$ to $H_j$. Then its inverse, $(A_j)^{-\frac{1}{2}}:H_j\to H_j^\prime$, is an isomorphism and, since $\{w_k^\prime\}_{k\in S_j}$ is an orthonormal basis of $H_j^\prime$, we have
	$$
	{\rm det}((A_j)^{\frac{1}{2}})=\wedge_{k\in S_j}A^\frac{1}{2}w_k^\prime=\wedge_{k\in S_j}w_k,
	$$
	which implies
	$$
	{\rm det}((A_j)^{-\frac{1}{2}})=\frac{1}{\wedge_{k\in S_j}w_k}.
	$$
	Moreover, since $(w_i^\prime)_{i=1}^n$ is an orthonormal basis of $\R^n$ we have that $\sum_{k\in S_j}w_k^\prime\otimes w_k^\prime=P_{H_j^\prime}$, therefore
	\begin{eqnarray*}
		\sum_{j=1}^m p_jP_{H_j^\prime}&=&\sum_{j=1}^m p_j\sum_{k\in S_j}w_k^\prime\otimes w_k^\prime=\sum_{i=1}^n\sum_{j=1}^mp_j\chi_{S_j}(i)w_i^\prime\otimes w_i^\prime\cr
		&=&\sum_{i=1}^nw_i^\prime\otimes w_i^\prime=I_n.
	\end{eqnarray*}
	Using Finner's inequality we have that, for any integrable functions $g_j:H_j^\prime\to[0,\infty)$,
	\begin{equation}\label{eq:GeometricBrascampLieb}
	\int_{\R^n}\prod_{j=1}^mg_j^{p_j}(P_{H_j^\prime}x)dx\leqslant\prod_{j=1}^m\left(\int_{H_j^\prime}g_j(x)dx\right)^{p_j}.
	\end{equation}
	It follows that, for any integrable functions $h_j: H_j\to[0,\infty)$, the functions $g_j=h_j\circ A_j^{\frac{1}{2}}:H_j^\prime\to[0,\infty)$ are integrable and satisfy
	$$
	\int_{H_j^\prime}g_j(x)dx=\int_{A_j^{-\frac{1}{2}}H_j}h_j(A_j^\frac{1}{2}x)dx=\int_{H_j}h_j(x)dx\big|{\rm det}((A_j)^{-\frac{1}{2}})\big|=\frac{\int_{H_j}h_j(x)dx}{|\wedge_{k\in S_j}w_k|}
	$$
	and
	\begin{eqnarray*}
		\int_{\R^n}\prod_{j=1}^mg_j^{p_j}(P_{H_j^\prime}x)dx&=&\big|{\rm det}A^{\frac{1}{2}}\big|\int_{\R^n}\prod_{j=1}^mh_j^{p_j}(A_j^\frac{1}{2}P_{H_j^\prime}A^{\frac{1}{2}}x)dx\cr
		&=&|\wedge_{i=1}^nw_i|\int_{\R^n}\prod_{j=1}^mh_j^{p_j}(A^\frac{1}{2}P_{H_j^\prime}A^{\frac{1}{2}}x)dx.
	\end{eqnarray*}
	Combining the above we get
	$$
	\int_{\R^n}\prod_{j=1}^mh_j^{p_j}(B_jx)dx\leqslant\frac{1}{|\wedge_{i=1}^nw_i|\prod_{j=1}^m|\wedge_{k\in S_j}w_k|^{p_j}}\prod_{j=1}^m\left(\int_{H_j}h_j(x)dx\right)^{p_j},
	$$
	where $B_j=A^\frac{1}{2}P_{H_j^\prime}A^{\frac{1}{2}}=\sum_{k\in S_j}w_k\otimes w_k$. Notice that $B_j$ is an isomorphism from $H_j$ to $H_j$. Now, let $f_j:H_j\to[0,\infty)$ be integrable functions and set $h_j=f_j\circ(B_j)^{-1}$. We observe that for every $x\in\R^n$, we have  $B_jx=B_jP_{H_j}x$ and that, fixing an orthonormal basis in $H_j$, we can write $B_j=\sum_{k\in S_j}w_k\otimes w_k=M_jM_j^t$ where $M_j$ is the matrix whose columns are the vectors $(w_k)_{k\in S_j}$ written with respect to that orthonormal basis. This implies that ${\rm det}B_j=(\wedge_{k\in S_j}w_j)^2$ and
	$$
	\int_{\R^n}\prod_{j=1}^mf_j^{p_j}(P_{H_j}x)dx\leqslant\frac{\prod_{j=1}^m|\wedge_{k\in S_j}w_k|^{p_j}}{|\wedge_{i=1}^nw_i|}\prod_{j=1}^m\left(\int_{H_j}f_j(x)dx\right)^{p_j}.
	$$
	Since the constant in the reverse Brascamp-Lieb inequality is the inverse of the constant in Brascamp-Lieb inequality, see \cite{Bar}, we also obtain the reverse inequality.
\end{proof}

\textbf{Acknowledgements.}
Lemma 3.1 was proved jointly with Finlay Dupree Mcintyre. We would also like to thank Professor Apostolos Giannopoulos for helpful discussions.

	\bibliographystyle{plain}
	\bibliography{biblio}

 \end{document}